\documentclass[10pt]{amsart}
\usepackage{amsmath,amssymb,latexsym,esint,cite,mathrsfs}
\usepackage{verbatim,wasysym}
\usepackage[left=2.4cm,right=2.6cm,top=2.9cm,bottom=2.9cm]{geometry}
\usepackage{tikz,enumitem,graphicx, subfig, microtype, color}
\usepackage{epic,eepic}

\usepackage[colorlinks=true,urlcolor=blue, citecolor=red,linkcolor=blue,
linktocpage,pdfpagelabels, bookmarksnumbered,bookmarksopen]{hyperref}
\usepackage[hyperpageref]{backref}
\usepackage[english]{babel}

\numberwithin{equation}{section}

\newtheorem{thm}{Theorem}[section]
\newtheorem{lem}[thm]{Lemma}
\newtheorem{cor}[thm]{Corollary}
\newtheorem{Prop}[thm]{Proposition}

\newtheorem{Rem}[thm]{Remark}

\newcommand{\cat}{{\rm cat}}
\newcommand{\A}{\mathbb{A}}
\newcommand{\B}{\mathbb{B}}
\newcommand{\C}{\mathbb{C}}
\newcommand{\D}{\mathbb{D}}
\newcommand{\E}{\mathbb{E}}

\newcommand{\dist}{{\rm dist}}

\newcommand{\N}{\mathbb{N}}

\newcommand{\R}{\mathbb{R}}

\newcommand{\cal}{\mathcal}
\renewcommand{\cat}{{\rm cat}}
\renewcommand{\i}{{\rm i}}
\renewcommand{\O}{{\mathcal O}}

\def\vr {\varepsilon}

\begin{document}

\title{Singularly perturbed critical Choquard equations}

\author[C.O.\ Alves]{Claudianor O. Alves}
\author[F. Gao]{Fashun Gao}
\author[M.\ Squassina]{Marco Squassina}
\author[M.\ Yang]{Minbo Yang$^*$}

\address[C.O.\ Alves]{Unidad Acad\'emica de Matem\'aticas - UAMat \newline\indent
	Universidade Federal de Campina Grande, \newline\indent
	CEP: 58429-900, Campina Grande - PB, Brasil}
\email{coalves@dme.ufcg.edu.br}

\address[F.\ Gao]{Department of Mathematics\newline\indent
	Zhejiang Normal University, \newline\indent
	321004, Jinhua, Zhejiang, China}
\email{fsgao@zjnu.edu.cn}

\address[M.\ Squassina]{Dipartimento di Matematica e Fisica \newline\indent
	Universit\`a Cattolica del Sacro Cuore, \newline\indent
	Via dei Musei 41, I-25121 Brescia, Italy}
\email{marco.squassina@unicatt.it}

\address[M.\ Yang]{Department of Mathematics\newline\indent
	Zhejiang Normal University, \newline\indent
	321004, Jinhua, Zhejiang, China}
\email{mbyang@zjnu.edu.cn}


\thanks{Claudianor O. Alves is partially supported by CNPq/Brazil
	304036/2013-7;\\
Marco Squassina is member of the Gruppo
	Nazionale per l'Analisi Matematica, la Probabilit\`a e le loro Applicazioni;\\
 $^*$Minbo Yang is the corresponding author who is partially supported by NSFC (11571317) and ZJNSF (LY15A010010)}

\begin{abstract}
	In this paper we study the semiclassical limit for the singularly perturbed Choquard equation
	$$
	-\vr^2\Delta u +V(x)u =\vr^{\mu-3}\Big(\int_{\R^3} \frac{Q(y)G(u(y))}{|x-y|^\mu}dy\Big)Q(x)g(u)  \quad \mbox{in $\R^3$},
	$$
	where $0<\mu<3$, $\vr$ is a positive parameter, $V,Q$ are two continuous real function on
	$\R^3$ and $G$ is the primitive of $g$ which is of critical growth due to the Hardy-Littlewood-Sobolev inequality. Under suitable assumptions on $g$, we first establish the existence of ground states for the critical Choquard equation with constant coefficients. Next we establish existence and multiplicity of semi-classical solutions and characterize the concentration behavior by variational methods.
\end{abstract}

\maketitle

\begin{center}
	\begin{minipage}{8.5cm}
		\small
		\tableofcontents
	\end{minipage}
\end{center}
%

\section{Introduction and results}
The stationary Choquard equation
\begin{equation*}
 -\Delta u +V(x)u =\Big(\int_{\R^N} \frac{|u(y)|^p}{|x-y|^\mu}dy\Big)|u|^{p-2}u,  \qquad \mbox{in $\R^N$},
\end{equation*}
where $N\geq3$,  $0<\mu<N$, arises in many interesting physical situations in quantum theory and plays an important
role in the theory of Bose-Einstein condensation where it accounts for the finite-range many-body interactions. For $N=3$, $p=2$ and $\mu=1$, it was investigated by Pekar in \cite{P1} to study the quantum theory of a polaron at rest. In \cite{Li1}, Choquard applied it as approximation to Hartree-Fock theory of one-component plasma. This equation was also proposed by Penrose in \cite{MPT} as a model of selfgravitating matter and is known in that context as the Schr\"odinger-Newton equation. For a complete and updated
discussion upon the current literature of such problems, we refer the interested reader to the guide \cite{guide}. We also mention \cite{fract}, where the fractional case is treated.

In the present paper we are interested in the existence, multiplicity and concentration behavior of the semi-classical solutions of the singularly perturbed nonlocal elliptic equation
\begin{equation}\label{SCP}
-\vr^2\Delta u +V(x)u =\vr^{\mu-3}\Big(\int_{\R^3} \frac{Q(y)G(u(y))}{|x-y|^\mu}dy\Big)Q(x)g(u),  \qquad \mbox{in $\R^3$},
\end{equation}
where $0<\mu<3$, $\vr$ is a positive parameter, $V$, $Q$ are real continuous functions on $\R^3$. As $\varepsilon$ goes to zero in \eqref{SCP}, the existence and asymptotic behavior of the solutions of the singularly perturbed equation \eqref{SCP} is known as the {\em semi-classical problem}. It was used to describe the transition between of
Quantum Mechanics and classical Mechanics. For the local Schr\"{o}dinger equation
 \begin{equation*}
 -\vr^2\Delta u +V(x)u  =g(u)  \quad \mbox{in $\R^N$},
\end{equation*}
it goes back to the pioneer work \cite{FW} by Floer and Weinstein. Since then, it has been studied
extensively under various hypotheses on the potential and the nonlinearity, see for example
\cite{ABC, DXL, FW, JT, DF2, DF1, R, WX} and
the references therein. Particularly, the existence and concentration of solutions for local Schr\"{o}dinger equation with critical exponent was investigated in \cite{ADS, BZZ, DL, ZCZ}.  For a Schr\"{o}dinger equation of the form
\begin{equation}\label{S.S2}
 -\varepsilon^2\Delta u +V(x)u  =K(x)u^{r-1}+Q(x)u^{t-1}  \quad \mbox{in $\mathbb{R}^N$},
\end{equation}
Wang and Zeng \cite{WZ} proved that
the concentration points are located on the middle ground of the competing potential functions and in some cases are given explicitly in terms of these functions.
Cingolani and Lazzo \cite{CL} obtained a multiplicity
result involving the set of global minima of a function which  provides some kind of global median value between the minimum of $V$
and the maximum of $K$ and $Q$. We also mention the paper \cite{AMSS} by Ambrosetti, Malchiodi
and Secchi where the authors  considered the case $Q=0$. Among other results, they
proved that the number of solutions of \eqref{S.S2} is related with the set of minima of
a function given explicitly in terms of $V, K, r $, and the dimension $N$.  Ding and Liu \cite{DXL} considered
$$\aligned
\left(-\i\varepsilon\nabla + A(x)\right)^2 u  +V(x)u =Q(x)\big(g(|u|)+|u|^{2^*-2}\big)u,
\endaligned
$$
for $u\in H^1_A(\R^N, \mathbb{C})$, where
$A:\mathbb{R}^N \to\mathbb{R}^N$ denotes a continuous magnetic potential associated with a magnetic field $B$, $g(|u|)u$ is a superlinear and subcritical. Under suitable  assumptions on the potentials, the authors obtained some new concentration phenomena of the semi-classical ground states.
 It can be observed that if $u$ is a solution of the nonlocal equation (\ref{SCP}), for $x_0\in\R^N$, the function $v=u(x_0+\varepsilon x)$ satisfies
 $$
-\Delta v +V(x_0+\vr x)v  =\Big(\int_{\R^3} \frac{Q(x_0+\vr y)G(v(y))}{|x-y|^\mu}dy\Big)Q(x_0+\vr x)g(v) \qquad \mbox{in $\R^3$}.
 $$
 It suggests some convergence, as $\vr\to0$, of the family of solutions to a solution $u_0$ of the limit problem
\begin{equation*}
-\Delta v +V(x_0)v=Q^2(x_0)\Big(\int_{\R^3} \frac{G(v(y))}{|x-y|^\mu}dy\Big)g(v) \qquad \mbox{in $\R^3$}.
\end{equation*}
Hence we know that the equation
\begin{equation}\label{CC}
-\Delta u +u  =\Big(\int_{\R^3} \frac{G(v(y))}{|x-y|^\mu}dy\Big)g(u) \qquad \mbox{in $\R^{3}$}
\end{equation}
plays the role of limit equation in the study of the semiclassical problems. To apply the Lyapunov-Schmidt reduction techniques, it relies a lot on the uniqueness and non-degeneracy of the ground states of the limit problem which is not completely known for the ground states of the nonlocal Choquard equation \eqref{CC}. There is a considerable amount of work on investigating the
properties of this type equation. We refer to Lieb \cite{Li1} and
Lions \cite{Ls} for the existence  and uniqueness of
positive solutions to equation
\begin{equation}\label{AP}
 -\Delta u +u =\Big(\int_{\R^3} \frac{|u(y)|^2}{|x-y|}dy\Big)u   \qquad \mbox{in $\R^3$}.
\end{equation}
Recently, by using
the method of moving planes,  Ma and Zhao \cite{ML} proved that all
the positive solutions of  equation \eqref{AP} must be
radially symmetric and monotone decreasing about some fixed point. Which means that the positive solution of equation \eqref{AP}
is uniquely determined up to translations. Especially, they studied
the classification of all positive solutions to the generalized
nonlinear Choquard problem
\begin{eqnarray}\label{limit equation3}
-\Delta u+u=\Big(\int_{\R^N} \frac{|u(y)|^p}{|x-y|^\mu}dy\Big)|u|^{p-2}u,
\end{eqnarray}
under some assumptions on $\mu$, $p$ and $N$,  they proved that
all the positive solutions of \eqref{limit equation3} must be radially
symmetric and monotone decreasing about some fixed point.
In \cite{MS1}, Moroz and Van
Schaftingen completely investigated the qualitative properties of solutions of \eqref{limit equation3} and showed the regularity, positivity
and radial symmetry decay behavior at infinity.  The authors also considered in \cite{MS2, MS4} the
existence of ground states under the assumption of Berestycki-Lions type and studied the existence of solutions for the nonlocal equation with lower critical exponent due to the Hardy-Littlewood-Sobolev inequality.
For $N = 3$, $\mu=1$ and $F(s)=|s|^2$,  the uniqueness and non-degeneracy of the ground states were proved in Lenzmann, Wei
and Winter in \cite{Len, WW}. Wei
and Winter also constructed families of solutions
 by a Lyapunov-Schmidt type reduction when $\inf V>0$ and $Q(x)=1$.
 Cingolani et.al. \cite{CCS} applied the penalization arguments due to Byeon and Jeanjean \cite{BJ} and showed that there exists a family of solutions
having multiple concentration regions which are located around the minimum points of the potential. For any $N\geq 3$ and $F(u)=u^p$ with $\frac{2N-\mu}{N}\leq p<\frac{2N-\mu}{N-2}$ in \eqref{SCP}, Moroz and Van Schaftingen in \cite{MS3} developed a nonlocal penalization technique and showed that equation \eqref{SCP} has a family of solutions
concentrating around the local minimum of $V$ with $V$ satisfying some additional assumptions at infinity. In \cite{AY1, AY2}, Alves and Yang proved the existence, multiplicity and concentration of solutions for the equation by penalization method and Lusternik-Schnirelmann theory.
The planar case was considered by \cite{ACTY}, where the authors first established the existence of ground state for the limit problem with critical exponential growth and then studied the concentration around the global minimum set. In \cite{YD}, Yang and Ding considered the equation
 $$
-\varepsilon^2\Delta u +V(x)u  =\Big(\int_{\R^3} \frac{u^p(y)}{|x-y|^\mu}dy\Big)u^{p-1} \quad \mbox{in $\R^3$},
$$
and they obtained the existence of solutions which goes to $0$ with suitable parameter $p$ and $\mu$. To study problem \eqref{SCP} variationally the following Hardy-Littlewood-Sobolev inequality \cite{LL}
is the starting point.
\begin{Prop}[Hardy-Littlewood-Sobolev inequality]
	\label{HLS}
 Let $t,r>1$ and $0<\mu<3$ with
 $$
 1/t+\mu/3+1/r=2,
 $$
 $f\in L^{t}(\mathbb{R}^3)$ and $h\in L^{r}(\mathbb{R}^3)$. There exists a sharp constant $C(t,\mu,r)$, independent of $f,h$, such that
\begin{equation}\label{HLS1}
\int_{\mathbb{R}^{3}}\int_{\mathbb{R}^{3}}\frac{f(x)h(y)}{|x-y|^{\mu}}dxdy\leq C(t,\mu,r) |f|_{t}|h|_{r}.
\end{equation}
If $t=r=6/(6-\mu)$, then
$$
 C(t,\mu,r)=C(\mu)=\pi^{\frac{\mu}{2}}\frac{\Gamma(\frac{3}{2}-\frac{\mu}{2})}{\Gamma(3-\frac{\mu}{2})}\left\{\frac{\Gamma(\frac{3}{2})}{\Gamma(3)}\right\}^{-1+\frac{\mu}{3}}.
$$
In this case there is equality in \eqref{HLS1} if and only if $f\equiv Ch$ and
$$
h(x)=A(\gamma^{2}+|x-a|^{2})^{-(6-\mu)/2}
$$
for some $A\in \mathbb{C}$, $\gamma\in\mathbb{R}\setminus\{0\}$ and $a\in \mathbb{R}^{3}$.
\end{Prop}

\noindent
Notice that, by the Hardy-Littlewood-Sobolev inequality, for $u\in  H^1(\R^3)$, the integral
$$
\int_{\mathbb{R}^{3}}\int_{\mathbb{R}^{3}}\frac{|u(x)|^{t}|u(y)|^{t}}{|x-y|^{\mu}}dxdy
$$
is well defined if
$$
\frac{6-\mu}{3}\leq t\leq 6-\mu.
$$
Thus $(6-\mu)/3$ is called the lower critical exponent and $6-\mu$ is the upper critical exponent due to the Hardy-Littlewood-Sobolev inequality, we also recall that $6$ is critical Sobolev exponent for dimension 3.
As to the best knowledge of us, the existing results for the existence and concentration behavior of solutions for the Choquard equation were obtained under the subcritical growth assumption, namely $t<6-\mu$. It is then quite natural to ask if the nonlinearity $g(u)$ in equation \eqref{CC} is of upper critical growth in the sense of the Hardy-Littlewood-Sobolev inequality, does the ground state solution still exist? Furthermore, can we establish the existence and multiplicity results for the singular perturbed critical Choquard equation \eqref{SCP} and characterize the concentration phenomena around the minimum set of linear potential $V(x)$ or the maximum set of the nonlinear potential $Q(x)$? In the present paper we are going to answer the above questions and to investigate the existence, multiplicity and the concentration behavior of the solutions of the Choquard equation with critical exponent due to the Hardy-Littlewood-Sobolev inequality.

The first goal of the present paper is to study the existence of nontrivial solution for the critical Choquard equation of the form
\begin{equation}\label{A}
\left\{\begin{array}{l}
\displaystyle-\Delta u+\kappa u
=\left(\int_{\mathbb{R}^3}\frac{\nu|u(y)|^{6-\mu}+\tau F(u(y))}{|x-y|^{\mu}}dy\right)\Big(\nu|u|^{4-\mu}u+ \frac{\tau}{6-\mu}f(u)\Big)\hspace{4.14mm}\mbox{in}\hspace{1.14mm} \R^{3},\\
\displaystyle u\in H^{1}(\R^{3}),
\end{array}
\right.
\end{equation}
where  $0<\mu<3$, $\kappa,\nu,\tau$ are positive constants and
$$
F(u)=\displaystyle\int_{0}^{u}f(s)ds.
$$
Since we are interested in the existence of positive solutions, we shall suppose that $f : \R^+\to \R$ verifies the following conditions.
\vskip2pt
\noindent
There exists $p,q,\zeta$ such that
$$
\frac{6-\mu}{3}<q \leq p<6-\mu,\quad \ 5-\mu<\zeta<6-\mu,
$$
and $c_{0},c_{1}> 0$ such that for all $s\in\mathbb{R}$:
$$
|f(s)|\leq c_{0}(|s|^{q-1}+|s|^{p-1})\ \   \hbox{and } \ \  F(u)\geq c_1|s|^\zeta.\eqno{(f_1)}
$$
We suppose that $f$ verifies the Ambrosetti-Rabinowitz type condition for nonlocal problems:
there is $\alpha>2$ with
$$
0<\alpha F(s)\leq 2f(s)s, \quad \forall s \in\R^+. \eqno{(f_2)}
$$
Moreover
$$
\{\displaystyle s\to {f(s)}\}\ \ \hbox{ is strictly increasing on } \R^+. \eqno{(f_3)}
$$

\begin{Rem} \label{ff}
From $(f_3)$ and $(f_2)$, we know there is $\varsigma >0$ such that
$$
f'(s)s^2 -\left(1+\varsigma-\frac{\alpha}{2}\right)f(s)s> 0,\quad  \forall s \in\R^+. \eqno{(f_4)}
$$
In fact, since $f'(s)>0$ and $f(s)s>0$, by taking $\varsigma=\frac12(\frac{\alpha}{2}-1)$, we obtain $(f_4)$ immediately.
\end{Rem}
\noindent
The first result is about the existence of ground state for the autonomous case, that is
 \begin{thm}[Existence of ground states]
 	\label{AQ}
Suppose that $(f_1)$-$(f_3)$ hold. Then, for any $\kappa,\nu,\tau>0$, \eqref{A} admits a ground state solution.
\end{thm}

 \begin{Rem}\rm
As observed for the local Schr\"{o}dinger equation, the subcritical perturbation is necessary to secure the existence of a nontrivial solution. Note that if $f(u)=0$ and $u$ is a solution, we can establish the following $Poho\check{z}aev$ identity
$$
\frac{1}{2}\int_{\R^3} |\nabla u|^{2}dx+\frac{3\kappa}{2}\int_{\R^3} |u|^{2}dx=\frac{\nu^2}{2}\int_{\R^3}
\int_{\R^3}\frac{|u(x)|^{6-\mu}|u(y)|^{6-\mu}}{|x-y|^{\mu}}dxdy.
$$
Then $\kappa\int_{\R^3} |u|^{2}dx=0$,
which means there are no solutions with $\kappa \neq 0$.
\end{Rem}

Next we are going to study the existence of semi-classical solutions
with concentration around the global maximum of $Q(x)$. For simplicity, we assume that $V(x)=1$ and consider
\begin{equation}\label{SCP1}
-\vr^2\Delta u +u =\vr^{\mu-3}\left(\int_{\mathbb{R}^3}\frac{Q(y)(|u(y)|^{6-\mu}+F(u(y)))}{|x-y|^{\mu}}dy\right)\Big(Q(x)\Big(|u|^{4-\mu}u+ \frac{1}{6-\mu}f(u)\Big)\Big)  \quad \mbox{in} \quad \R^3,
\end{equation}

We denote
$$
\nu_{\max}:=\max_{x\in\R^3} Q(x), \ \ \mathcal{Q}:=\{x\in \R^3: Q(x)=\nu_{\max}\},\ \ \nu_{\max}>\nu_{\infty}=\limsup_{|x|\to \infty}Q(x) \eqno{(Q)}
$$
and suppose that $Q:\R^3\to \R$ is a bounded continuous function with $\inf_{x\in\R^3} Q(x)>0$.
For this case we have the following theorem

\begin{thm}[Semiclassical limit I: concentration around maxima of $Q$]
	\label{Result1}
Suppose that the nonlinearity $f$ satisfies $(f_1)$-$(f_3)$ and the potential function $Q$ satisfies condition $(Q)$. Then, for any $\vr>0$, equation \eqref{SCP1} has at least one positive ground state solution $u_\vr$. Moreover, the following facts hold:
\begin{itemize}
\item[$(a)$] There exists a maximum point $x_\vr\in\R^3$  of $u_\vr$, such that
$$
\lim_{\vr\to0}\dist(x_\vr, \mathcal{Q})=0,
$$
and for some $c,C>0$, \\
$$
|u_\vr(x)|\leq C\exp\Big(-\frac{c}{\vr}|x-x_\vr|\Big).
$$
\item[$(b)$] Setting $v_\vr(x):=u_\vr(\vr x+x_\vr)$, for any sequence $x_\vr\to x_0$, $\vr\to0$, $v_\vr$ converges in $H^{1}(\R^3)$ to a ground state solution $v$ of
 $$
-\Delta v +v  =\nu^2_{\max}\left(\int_{\mathbb{R}^3}\frac{|v(y)|^{6-\mu}+F(v(y))}{|x-y|^{\mu}}dy\right)\Big(|v|^{4-\mu}v+ \frac{1}{6-\mu}f(v)\Big).
$$
\end{itemize}
\end{thm}

Finally we are going to study the existence, multiplicity of semiclassical solutions that concentrating around the global minimum of $V(x)$. We are going to study
\begin{equation}\label{SCP2}
-\vr^2\Delta u+V(x)u =\vr^{\mu-3}\left(\int_{\mathbb{R}^3}\frac{|u(y)|^{6-\mu}+F(u(y))}{|x-y|^{\mu}}dy\right)\Big(|u|^{4-\mu}u+ \frac{1}{6-\mu}f(u)\Big)  \quad \mbox{in} \quad \R^3.
\end{equation}

\noindent
Assume that $V:\R^3\to \R$ is a bounded continuous function satisfying:
$$
0<\kappa_{\min}:=\min_{x\in\R^3} V(x), \ \ \mathcal{V}:=\{x\in \R^3: V(x)=\kappa_{\min}\},\ \
\kappa_{\min}<\kappa_{\infty}=\liminf_{|x|\to \infty}V(x)<\infty.\eqno{(V)}
$$
This
kind of hypothesis was introduced by Rabinowitz in \cite{R}.

\begin{thm}[Semiclassical limit II: concentration around minima of $V$]
	\label{Result2}
Suppose that the nonlinearity $f$ satisfies
$(f_1)$-$(f_3)$ and the potential function $V(x)$ satisfies condition $(V)$.
Then, for any $\vr>0$, equation \eqref{SCP2} has at least one positive ground state solution $u_\vr$. Moreover,
\begin{itemize}
\item[$(a)$] There exists a maximum point $x_\vr\in\R^3$  of $u_\vr$, such that
$$
\lim_{\vr\to0}\dist(x_\vr, \mathcal{V})=0,
$$
and for some $c,C>0$,
$$
|u_\vr(x)|\leq C\exp\Big(-\frac{c}{\vr}|x-x_\vr|\Big).
$$
\item[$(b)$] Setting $v_\vr(x):=u_\vr(\vr x+x_\vr)$, for any sequence $x_\vr\to x_0$, $\vr\to0$, $v_\vr$ converges in $H^{1}(\R^3)$ to a ground state solution $v$ of
$$
-\Delta v +\kappa^2_{\min}v  =\left(\int_{\mathbb{R}^3}\frac{|v(y)|^{6-\mu}+F(v(y))}{|x-y|^{\mu}}dy\right)\Big(|v|^{4-\mu}v+ \frac{1}{6-\mu}f(v)\Big).
$$
\end{itemize}
\end{thm}

\noindent
The multiplicity of solutions for the nonlocal problem can be characterized by the Lusternik-Schnirelman category of
the sets $\mathcal{V}$ and $\mathcal{V}_\delta$ defined by
 $$
\mathcal{V}_\delta=\{x\in\R^3:\dist(x,\mathcal{V})\leq\delta\},\  \hbox{for}\  \delta>0.
$$

\begin{thm}[Multiplicity of solutions]
	\label{Result3}
Suppose that the nonlinearity $f$ satisfies $(f_1)$-$(f_3)$ and the potential function $V(x)$ satisfies condition $(V)$. Then, for any $\delta>0$, there exists $\varepsilon_\delta$ such that equation \eqref{SCP2} has at least $\cat_{\mathcal{V}_\delta}(\mathcal{V})$ positive solutions, for any $0<\varepsilon<\varepsilon_\delta$. Moreover,  let $u_\vr$ denotes one of these positive solutions with $\eta_\vr\in\R^3$ its global maximum. Then
$$
\lim_{\varepsilon\to 0}V(\eta_\varepsilon)=\kappa_{\min}.
$$
\end{thm}

\vskip3pt
\noindent
{\bf Basic notations:} \\
\noindent $\bullet$ $C$, $C_i$ denote positive constants.\\
\noindent $\bullet$ $B_R$ denote the open ball centered at the origin with
radius $R>0$. \\
\noindent $\bullet$ $C_0^{\infty}(\R^3)$ denotes  the space of the functions
infinitely differentiable with compact support in $\R^3$. \\
\noindent $\bullet$ For a mensurable function $u$, we denote by $u^{+}$ and $u^{-}$ its positive and negative parts respectively, given by
$$
u^{+}(x)=\max\{u(x),0\} \quad \mbox{and} \quad u^{-}(x)=\min\{u(x),0\}.
$$

\noindent $\bullet$ $E:=H^{1}(\R^3)$ is the usual Sobolev space with norm
   $$
    \|u\|:=\left(\int_{\R^3}(|\nabla u|^2+ |u|^2)dx\right)^{1/2}.
  $$
\noindent $\bullet$ $L^s(\R^3)$, for $1 \leq s <\infty$,
denotes the Lebesgue space with the norms
$$
| u |_s:=\Big(\int_{\R^3}|u|^sdx\Big)^{1/s}.
$$
\noindent $\bullet$ From the assumption on $V$, it follows that
 $$
    \|u\|_{\varepsilon}:=\left(\int_{\R^3}(|\nabla u|^2+ V(\varepsilon x)|u|^2)dx\right)^{1/2}
 $$
is an equivalent norm on $E$. \\

\noindent $\bullet$ \, Let $X$ be a real Hilbert space and $I:X \to \R$ be a functional of class ${C}^1$.
 We say that $(u_n)\subset X$ is a  Palais-Smale ($(PS)$ for short) sequence at $c$ for $I$ if $(u_n)$ satisfies
$$
I(u_n)\to c \,\,\, \mbox{and} \,\,\,\, I'(u_n)\to0,  \,\,\, \mbox{as} \,\,\, n\to\infty.
$$
Moreover, $I$ satisfies the $(PS)$ condition at $c$, if any $(PS)$ sequence at $c$ possesses a convergent subsequence.

\section{Autonomous critical equation}
Since there are no existing results for the nonlocal Choquard equation with upper critical exponent in the whole space, then we are going to study firstly the existence and properties of the ground state solutions of the autonomous equation \eqref{A} which will play the role of limit problem for the equation \eqref{SCP1} and \eqref{SCP2}. Let
$
G(u)=\nu|u|^{6-\mu}+\tau F(u)
$
and
$
g(u)=\frac{dG(u)}{du}
$, then equation \eqref{A} can be rewritten as
\begin{equation*}
\displaystyle-\Delta u+\kappa u
=\frac{1}{6-\mu}\left(\int_{\mathbb{R}^3}\frac{G(u(y))}{|x-y|^{\mu}}dy\right)g(u)\hspace{4.14mm}\mbox{in}\hspace{1.14mm} \R^{3}.
\end{equation*}

\begin{Rem}\label{R1}\rm
Assumption $(f_2)$ with $\alpha>2$ implies the existence of constant $\theta>2$ such that  Ambrosetti-Rabinowitz condition for $g$ holds: for every $u\neq0$, $0<\theta G(u)\leq 2ug(u)$.
Moreover, we also assume that there is $\varsigma >0$ such that
$$
g'(s)s^2 -\left(1+\varsigma-\frac{\theta}{2}\right)g(s)s> 0  \,\,\,\ \forall s >0. \eqno{(f'_4)}
$$
\end{Rem}
\noindent
For all $u\in D^{1,2}(\mathbb{R}^3)$ we know
$$
\Big(\int_{\R^3}\int_{\R^3}\frac{|u(x)|^{6-\mu}|u(y)|^{6-\mu}}{|x-y|^{\mu}}dxdy\Big)^{\frac{1}{6-\mu}}\leq C(3,\mu)^{\frac{1}{6-\mu}}|u|_{6}^{2},
$$
where $C(3,\mu)$ is defined as in Proposition \ref{HLS}.
We use $S_{H,L}$ to denote the best constant defined by
\begin{equation}\label{S1}
S_{H,L}:=\displaystyle\inf\limits_{u\in D^{1,2}(\R^3)\backslash\{{0}\}}\ \ \frac{\displaystyle\int_{\R^3}|\nabla u|^{2}dx}{\left(\displaystyle\int_{\R^3}\int_{\R^3}\frac{|u(x)|^{6-\mu}|u(y)|^{6-\mu}}{|x-y|^{\mu}}dxdy\right)^{\frac{1}{6-\mu}}}.
\end{equation}

\noindent
From the comments above, we can easily draw the following conclusion.

\begin{lem}[Optimizers for $S_{H,L}$]\cite{GY}
	\label{ExFu}
The constant $S_{H,L}$ defined in \eqref{S1} is achieved if and only if
$$
u=C\left(\frac{b}{b^{2}+|x-a|^{2}}\right)^{\frac{1}{2}} ,
$$
where $C>0$ is a fixed constant, $a\in \mathbb{R}^{3}$ and $b>0$ are parameters. Moreover,
\begin{equation}
\label{legame}
S_{H,L}=\frac{S}{C(3,\mu)^{\frac{1}{6-\mu}}},
\end{equation}
where $S$ is the best Sobolev constant of the continuous embedding $D^{1,2}(\mathbb{R}^3) \hookrightarrow L^{2^{*}}(\mathbb{R}^3).$ In particular, let
$$
U(x)=\frac{C_0}{(1+|x|^{2})^{\frac{1}{2}}}
$$
be a minimizer for $S$ which satisfies $-\Delta U=U^5$, then
\begin{equation*}
\tilde{U}(x)=S^{\frac{(\mu-3)}{4(5-\mu)}}C(3,\mu)^{\frac{-1}{2(5-\mu)}}U(x),
\end{equation*}
is the unique  minimizer for $S_{H,L}$ that satisfies
$$
-\Delta u=\left(\int_{\R^3}\frac{|u(y)|^{6-\mu}}{|x-y|^{\mu}}dy\right)|u|^{4-\mu}u\ \ \   \hbox{in}\ \ \ \R^3.
$$
\end{lem}
\begin{proof}
We sketch the proof for the completeness of the paper. By the Hardy-Littlewood-Sobolev inequality, we have
$$
S_{H,L}\geq\frac{1}{C(3,\mu)^{\frac{1}{6-\mu}}}\inf\limits_{u\in D^{1,2}(\R^3)\backslash\{{0}\}}\ \ \frac{\displaystyle\int_{\R^3}|\nabla u|^{2}dx}{|u|_{6}^{2}}
=\frac{S}{C(3,\mu)^{\frac{1}{6-\mu}}}.
$$
On the other hand, the equality in the Hardy-Littlewood-Sobolev inequality \eqref{HLS1} holds if and only if
$$
f(x)=h(x)=C\left(\frac{b}{b^{2}+|x-a|^{2}}\right)^{\frac{6-\mu}{2}},
 $$
where $C>0$ is a fixed constant, $a\in \mathbb{R}^{3}$ and $b\in(0,\infty)$ are parameters. Thus
$$
\Big(\int_{\R^3}\int_{\R^3}\frac{|u(x)|^{6-\mu}|u(y)|^{6-\mu}}{|x-y|^{\mu}}dxdy\Big)^{\frac{1}{6-\mu}}= C(3,\mu)^{\frac{1}{6-\mu}}|u|_{6}^{2},
$$
if and only if
$$
\bar u(x)=C\left(\frac{b}{b^{2}+|x-a|^{2}}\right)^{\frac{1}{2}}.
 $$
Then, by the definition of $S_{H,L}$, we know
$$
S_{H,L}\leq \frac{\displaystyle\int_{\R^3}|\nabla \bar u(x)|^{2}dx}{\left(\displaystyle\int_{\R^3}\int_{\R^3}\frac{|\bar u(x)|^{6-\mu}|\bar u(y)|^{6-\mu}}{|x-y|^{\mu}}dxdy\right)^{\frac{1}{6-\mu}}}
=\frac{1}{C(3,\mu)^{\frac{1}{6-\mu}}}\frac{\displaystyle\int_{\R^3}|\nabla \bar u|^{2}dx}{|\bar u|_{6}^{2}}.
$$
It is well-known that  $\bar u$ is a minimizer for $S$, thus we get
$$
S_{H,L}\leq \frac{S}{C(3,\mu)^{\frac{1}{6-\mu}}}.
$$
From the arguments above, we know that $S_{H,L}$ on $\bar u$ and \eqref{legame} holds.
By a simple calculation, we know
\begin{equation*}
\tilde{U}(x)=S^{\frac{(\mu-3)}{4(5-\mu)}}C(3,\mu)^{\frac{-1}{2(5-\mu)}}U(x)\\
\end{equation*}
is the unique  minimizer for $S_{H,L}$ that satisfies
$$
-\Delta u=\left(\int_{\R^3}\frac{|u(y)|^{6-\mu}}{|x-y|^{\mu}}dy\right)|u|^{4-\mu}u\ \ \ \hbox{in $\R^3$},
$$
and, moreover,
$$
\int_{\R^3}|\nabla \tilde{U}|^{2}dx=\int_{\R^3}\int_{\R^3}\frac{|\tilde{U}(x)|^{6-\mu}|\tilde{U}(y)|^{6-\mu}}{|x-y|^{\mu}}dxdy=S_{H,L}^{\frac{6-\mu}{5-\mu}},
$$
which concludes the proof.
\end{proof}

\noindent
Next, repeat the proof in \cite{GY}, we have one more important information about the best constant $S_{H,L}$.
\begin{lem}
 For every open subset $\Omega\subset \R^3$, we have
\begin{equation*}
S_{H,L}(\Omega):=\displaystyle\inf\limits_{u\in D_{0}^{1,2}(\Omega)\backslash\{{0}\}}\ \ \frac{\displaystyle\int_{\Omega}|\nabla u|^{2}dx}{\left(\displaystyle\int_{\Omega}\int_{\Omega}\frac{|u(x)|^{6-\mu}|u(y)|^{6-\mu}}{|x-y|^{\mu}}dxdy\right)^{\frac{1}{6-\mu}}}=S_{H,L},
\end{equation*}
and $S_{H,L}(\Omega)$ is never achieved except for $\Omega=\R^3$.
\end{lem}
\begin{proof}
Clearly $S_{H,L}\leq S_{H,L}(\Omega)$ by $D_{0}^{1,2}(\Omega)\subset D^{1,2}(\R^3)$. Let $(u_{n})\subset C_{0}^{\infty}(\R^3)$ be a minimizing sequence for $S_{H,L}$. We make translations and dilations for $(u_{n})$ by choosing $y_{n}\in\R^3$ and $\tau_{n}>0$ such that
$$
u_{n}^{y_{n},\tau_{n}}(x):=\tau_n^{1/2}u_{n}(\tau_{n} x+y_{n})\in C_{0}^{\infty}(\Omega),
$$
which satisfies
$$
\int_{\R^3}|\nabla u_{n}^{y_{n},\tau_{n}}|^{2}dx=\int_{\R^3}|\nabla u_{n}|^{2}dx
$$
and
$$
\int_{\Omega}\int_{\Omega}\frac{|u_{n}^{y_{n},\tau_{n}}(x)|^{6-\mu}|u_{n}^{y_{n},\tau_{n}}(y)|^{6-\mu}}{|x-y|^{\mu}}dxdy
=\int_{\R^3}\int_{\R^3}\frac{|u_{n}(x)|^{6-\mu}|u_{n}(y)|^{6-\mu}}{|x-y|^{\mu}}dxdy.
$$
Hence $S_{H,L}(\Omega)\leq S_{H,L}$. Moreover $S_{H,L}(\Omega)$ is never achieved except for $\Omega=\R^3$,
since $\tilde{U}(x)$ is the only class of functions with equality in the  Hardy-Littlewood-Sobolev inequality.
\end{proof}


\noindent
The energy functional associated to equation \eqref{A} is defined by
$$
\aligned
\Phi_{\kappa,\nu,\tau}(u)=\frac{1}{2}\int_{\R^3}(|\nabla u|^{2}+\kappa|u|^{2})dx-\frac{1}{2(6-\mu)}\int_{\R^3}
\int_{\R^3}\frac{G(u(y))G(u(x))}
{|x-y|^{\mu}}dxdy.
\endaligned
$$
From the growth assumptions on $f$, the Hardy-Littlewood-Sobolev inequality
implies that $\Phi_{\kappa,\nu,\tau}$ is well defined on $E$ and belongs to $C^1$ with its derivative given by
$$
\langle\Phi'_{\kappa,\nu,\tau}(u), \varphi\rangle=\int_{\R^3}(\nabla u\nabla\varphi +\kappa u\varphi)dx -\frac{1}{6-\mu}\int_{\R^3}\int_{\R^3}\frac{G(u(y))g(u(x))\varphi(x)}
{|x-y|^{\mu}}dxdy,\quad\,\,\, \forall u, \varphi \in E.
$$
Therefore, the solutions of equation \eqref{A} correspond to critical points of the energy $\Phi_{\kappa,\nu,\tau}$.
Let us denote by ${\cal N}_{\kappa,\nu,\tau}$ the Nehari manifold associated to $\Phi_{\kappa,\nu,\tau}$ defined by
$\{u\in E:u\neq0,\langle\Phi'_{\kappa,\nu,\tau}(u), u\rangle=0\},$ namely
$$
{\cal N}_{\kappa,\nu,\tau}=\left\{u\in E\setminus\{0\}: \int_{\R^3} (|\nabla u|^{2}+\kappa |u|^{2})dx=\frac{1}{6-\mu}\int_{\R^3}\int_{\R^3}\frac{G(u(y))g(u(x))u(x)}{|x-y|^{\mu}}dxdy\right\},
$$
or equivalently
${\cal N}_{\kappa,\nu,\tau}=\left\{u\in E\setminus\{0\}: \Psi_{\kappa,\nu,\tau}(u)=0\right\},$
where we have set
$$
\Psi_{\kappa,\nu,\tau}(u):=\int_{\R^3} (|\nabla u|^{2}+\kappa |u|^{2})dx-\frac{1}{6-\mu}\int_{\R^3}\int_{\R^3}\frac{G(u(y))g(u(x))u(x)}{|x-y|^{\mu}}dxdy.
$$
From $(f_3)$, for each $u\in E\backslash\{0\}$, there is a unique $t=t(u)>0$ such that
 $$
 \Phi_{\kappa,\nu,\tau}(t(u)u)=\max_{s\geq0}\Phi_{\kappa,\nu,\tau}(su),\qquad
  t(u)u\in {\cal N}_{\kappa,\nu,\tau}.
 $$
Furthermore, there exists $\delta>0$ such that
\begin{equation} \label{alpha}
\|u\|\geq \delta,\qquad  \forall u\in {\cal N}_{\kappa,\nu,\tau}.
\end{equation}
By assumption  $(f'_4)$, there exists $\varrho>0$ such that
$$
\langle\Psi'_{\kappa,\nu,\tau}(u), u\rangle\leq-\varrho, \quad \forall u \in {\cal N}_{\kappa,\nu,\tau} .
$$
In fact, from $(f_2)$ and $(f'_4)$, a direct computation yields
$$
\aligned
\langle\Psi'_{\kappa,\nu,\tau}(u), u\rangle&=2\int_{\R^3} (|\nabla u|^{2}+\kappa |u|^{2})dx\\
&-\frac{1}{6-\mu}\int_{\R^3}\int_{\R^3}\frac{g(u(y))u(y)g(u(x))u(x)+G(u(y))g'(u(x))u^2(x)+G(u(y))g(u(x))u(x)}{|x-y|^{\mu}}dxdy \\
&\leq\int_{\R^3}\int_{\R^3}\frac{G(u(y))[(1-\frac{\theta}{2})g(u(x))u(x)-g'(u(x))u^2(x)]}{|x-y|^{\mu}}dxdy\\
&\leq-\varsigma\int_{\R^3}\int_{\R^3}\frac{G(u(y))g(u(x))u(x)}{|x-y|^{\mu}}dxdy\,.
\endaligned
$$
Therefore, if there exists a sequence $(u_n) \subset {\cal N}_{\kappa,\nu,\tau} $ such that $\langle\Psi'_{\kappa,\nu,\tau}(u_n), u_n\rangle\to0$, then
would we have
$$
\int_{\R^3}\int_{\R^3}\frac{G(u_n(y))g(u_n(x))u_n(x)}{|x-y|^{\mu}} dxdy\to 0,
$$
and consequently $\|u_n\|\to 0$, which contradicts \eqref{alpha}.
Therefore, ${\cal N}_{\kappa,\nu,\tau}$ defines a natural manifold and, as it can be readily checked, minimizing $\Phi_{\kappa,\nu,\tau}$ over ${\cal N}_{\kappa,\nu,\tau}$
generates critical point of $\Phi_{\kappa,\nu,\tau}$.
\vskip3pt
\noindent
To get existence of nontrivial solution by Mountain Pass theorem, we need to check that
$\Phi_{\kappa,\nu,\tau}$ satisfies the Mountain Pass Geometry. For simplicity,  we let
$\kappa=\nu=\tau=1$ in the sequel. The following lemma is a revised one of the corresponding version in \cite{AF}
and we sketch here for the convenience of the readers.

\begin{lem}\label{mountain:1}
The functional $\Phi_{1,1,1}$ satisfies the Mountain Pass Geometry, that is,
\begin{itemize}
  \item[$(1)$] There exist $\rho, \delta_0>0$ such that $\Phi_{1,1,1}|_{\partial B_{\rho}}\geq\delta_0$ for all $u\in \partial B_{\rho}=\{u\in E:\|u\|=\rho\}$;
  \item[$(2)$]  There are $r>0$ and $e$ with $\|e\|>r$ such that $\Phi_{1,1,1}(e)< 0$.
\end{itemize}
\end{lem}
\begin{proof}
(1). From the growth assumptions on $f$ and the Hardy-Littlewood-Sobolev inequality, we derive
$$
\aligned
\Phi_{1,1,1}(u) &\geq  \frac12\|u\|^2-C(\|u\|^{2q}+\|u\|^{2(6-\mu)}),
\endaligned
$$
then $(1)$ follows if  $\rho$ is small enough.
(2). Fixed $u_0 \in E\setminus\{0\}$, we set
$\psi(t):=\Psi\big(\frac{tu_0}{\|u_0\|}\big)>0,$ where
\begin{equation*}
\displaystyle \Psi(u)=\frac{1}{2(6-\mu)}\int_{\R^3}\int_{\R^3}\frac{G(u(x))G(u(y))}{|x-y|^{\mu}}dxdy.
\end{equation*}
By the Ambrosetti-Rabinowitz condition $(f_2)$,
$$
\frac{\psi'(t)}{\psi(t)}\geq \frac{\alpha}{t},\quad \,\,\, \mbox{for} \,\,\, t>0.
$$
Integrating over the interval $(1, s\|u_0\|)$ where $s>\frac{1}{\|u_0\|}$, we find
$$
\Psi(su_0)\geq \Psi\Big(\frac{u_0}{\|u_0\|}\Big)\|u_0\|^{\alpha} s^{\alpha}.
$$
Therefore, for $s$ large we get
$$
 \Phi_{1,1,1}(su_0)\leq C_1 s^2-C_2s^{\alpha}.
$$
Since $\alpha>2$, $(2)$ follows with $e=s u_0$ and $s$ large enough.
\end{proof}

\noindent
By the Mountain Pass theorem without (PS) condition, there is a (PS) sequence $(u_n) \subset E$ such that
$$
\Phi'_{1,1,1}(u_n)\rightarrow0,\quad \Phi_{1,1,1}(u_n)\rightarrow
{m_{1,1,1}},
$$
where the minimax value $m_{1,1,1}$ can be characterized by
\begin{equation} \label{m}
m_{1,1,1}:=\inf_{u\in E\backslash\{0\}} \max_{t\geq 0}
\Phi_{1,1,1}(tu)= \inf_{u\in {\cal N}_{1,1,1}}\Phi_{1,1,1}(u).
\end{equation}

\noindent
By using the Ambrosetti-Rabinowitz condition, it is easy to see that $(u_n)$ is bounded in $E$. The next lemma establishes an important information involving $(PS)$ sequence which will be crucial later on.
\begin{lem}[Nonvanishing energy range]
	\label{PS1}
Assume that $(u_{n})\subset E$ is a $(PS)_{c}$-sequence with
\begin{equation*}
0<c<\frac{5-\mu}{2(6-\mu)}S_{H,L}^{\frac{6-\mu}{5-\mu}}.
\end{equation*}
Then $(u_{n})$ cannot be vanishing, namely there exist $r, \delta>0$ and a sequence $(y_n) \subset \R^3$ such that
\begin{equation*}
\liminf_{n\to\infty}\int_{B_r(y_n)}|u_n|^2dx\geq\delta.
\end{equation*}
\end{lem}
\begin{proof}
By contradiction, if $(u_{n})\subset E$ is vanishing, then \cite[Lemma 1.21]{MW} yields
$$
u_{n}\rightarrow0 \ \ \mbox{in} \ \ L^{r}(\R^3),
$$
as $n\to \infty$, where $2<r<6$. Choose $t,s$ close to $\frac{6}{6-\mu}$ satisfying
$$
1/t+\mu/3+1/s=2.
$$
Applying the Hardy-Littlewood-Sobolev inequality, we know
$$
\Big|\int_{\R^3}
\int_{\R^3}\frac{|u_{n}(y)|^{6-\mu}f(u_{n}(x))u_{n}(x)}
{|x-y|^{\mu}}dxdy\Big|\leq C|u_n|^{6-\mu}_{t(6-\mu)}(|u_n|^{p}_{sp}+|u_n|^{q}_{sq}),
$$
from where it follows
\begin{equation*}
\int_{\R^3}
\int_{\R^3}\frac{|u_{n}(y)|^{6-\mu}f(u_{n}(x))u_{n}(x)}
{|x-y|^{\mu}}dxdy \rightarrow0,  \ \  \ \ \int_{\R^3}
\int_{\R^3}\frac{|u_{n}(y)|^{6-\mu}F(u_{n}(x))}
{|x-y|^{\mu}}dxdy\rightarrow0,
\end{equation*}
as $n\to \infty$. Similarly,
\begin{equation*}
\int_{\R^3}
\int_{\R^3}\frac{F(u_n(y))f(u_{n}(x))u_{n}(x)}
{|x-y|^{\mu}}dxdy \rightarrow0,  \ \  \ \ \int_{\R^3}
\int_{\R^3}\frac{F(u_n(y))F(u_{n}(x))}
{|x-y|^{\mu}}dxdy\rightarrow0,
\end{equation*}
as $n\to \infty$. Then, since $(u_{n})$ is a Palais-Smale sequence for $\Phi_{1,1,1}$ with  $\Phi_{1,1,1}(u_{n})\rightarrow c$, we get
\begin{align}\label{D12}
c&=\frac{1}{2}\|u_{n}\|^{2}-\frac{1}{2(6-\mu)}\int_{\R^3}
\int_{\R^3}\frac{|u_{n}(x)|^{6-\mu}|u_{n}(y)|^{6-\mu}}
{|x-y|^{\mu}}dxdy+o_{n}(1) \\
\label{D13}
\|u_{n}\|^{2}&=\int_{\R^3}
\int_{\R^3}\frac{|u_{n}(x)|^{6-\mu}|u_{n}(y)|^{6-\mu}}
{|x-y|^{\mu}}dxdy+o_{n}(1).
\end{align}
If $\|u_{n}\|\rightarrow0$, then it follows from  \eqref{D12} and \eqref{D13} that $c=0$, which is a contradiction. Then $\|u_{n}\|\not\rightarrow0$
and, by virtue of formula \eqref{D13}, we obtain
\begin{equation}\label{D14}
\aligned
\|u_{n}\|^{2}\leq S_{H,L}^{\mu-6}\|u_{n}\|^{2(6-\mu)}+o_{n}(1).
\endaligned
\end{equation}
So in light of \eqref{D14} we get
$$
\liminf_{n\to\infty}\|u_{n}\|^2\geq S_{H,L}^{(6-\mu)/(5-\mu)}.
$$
Then from \eqref{D12} and \eqref{D13} we easily conclude that $c\geq (5-\mu)/(2(6-\mu))S_{H,L}^{(6-\mu)/(5-\mu)}$, contradiction the assumption.
Hence, there exist $r, \delta>0$ and $(y_n) \subset \R^3$ with
$\liminf_{n\to\infty}\int_{B_r(y_n)}|u_n|^2dx\geq\delta$.
\end{proof}

\begin{lem}\label{EMP1} Suppose that $(f_1)$-$(f_3)$ hold. Then there exists $u_{0}\in H^{1}(\R^3)\setminus\{0\}$  such that
\begin{equation*}
\sup_{t\geq0}\Phi_{1,1,1}(tu_{0})<\frac{5-\mu}{2(6-\mu)}S_{H,L}^{\frac{6-\mu}{5-\mu}}.
\end{equation*}
\end{lem}
\begin{proof}
For every $\varepsilon>0$, consider
$$
U_{\varepsilon}(x):=\sqrt{\varepsilon}U\left(\frac{x}{\varepsilon}\right),\quad
u_{\varepsilon}(x):=\psi(x)U_{\varepsilon}(x),\quad x\in\R^3,
$$
be the functions in formula \eqref{B17} in the Appendix, where
$\psi\in C_{0}^{\infty}(\R^3)$ is such that $\psi=1$ on $B(0,\delta)$ and $\psi=0$ on $\R^3\setminus B(0,2\delta)$ for some $\delta>0$.
From Lemma \ref{ECT} and  \cite[Lemma 1.46]{MW}, we know that
\begin{equation*}
\int_{\R^3}|u_{\varepsilon}|^{2^{\ast}} dx=C(3,\mu)^{\frac{3}{2(6-\mu)}}S_{H,L}^{\frac{3}{2}}+\O(\varepsilon^{3}),
\end{equation*}
\begin{align}
\label{B19}
\int_{\R^3}|\nabla u_{\varepsilon}|^{2}dx
&=C(3,\mu)^{\frac{3}{2(6-\mu)}}S_{H,L}^{\frac{3}{2}}+\O(\varepsilon), \\
\label{B191}
\int_{\R^3}|u_{\varepsilon}|^{2}dx&=
\displaystyle \O(\varepsilon), \\
\label{B201}
\int_{\R^3}\int_{\R^3}\frac{|u_{\varepsilon}(x)|^{6-\mu}|u_{\varepsilon}(y)|^{6-\mu}}
{|x-y|^{\mu}}dxdy
&\geq C(3,\mu)^{\frac{3}{2}}S_{H,L}^{\frac{6-\mu}{2}}-\O\left(\varepsilon^{\frac{6-\mu}{2}}\right), \\
\label{B20}
\int_{\R^3}\int_{\R^3}\frac{|u_{\varepsilon}(x)|^{\zeta}
|u_{\varepsilon}(y)|^{\zeta}}
{|x-y|^{\mu}}dxdy
&\geq \O(\varepsilon^{6-\mu-\zeta})-\O\left(\varepsilon^{\frac{6-\mu}{2}}\right),\quad  \hbox{if $\frac{6-\mu}{2}<\zeta<6-\mu$}.
\end{align}
Then the estimates \eqref{B19}-\eqref{B20} imply
$$\aligned
\Phi_{1,1,1}(tu_{\varepsilon})&\leq\frac{t^{2}}{2}\int_{\R^3}(|\nabla u_{\varepsilon}|^{2}+|u_{\varepsilon}|^{2})dx-\frac{t^{2(6-\mu)}}{2(6-\mu)}
\int_{\R^3}
\int_{\R^3}\frac{|u_{\varepsilon}(x)|^{6-\mu}|u_{\varepsilon}(y)|^{6-\mu}}
{|x-y|^{\mu}}dxdy\\
&\hspace{7mm}-\frac{t^{2\zeta}}{2(6-\mu)}c_1^2
\int_{\R^3}
\int_{\R^3}\frac{|u_{\varepsilon}(x)|^{\zeta}|u_{\varepsilon}(y)|^{\zeta}}
{|x-y|^{\mu}}dxdy\\
&\leq\frac{t^{2}}{2}\big(C(3,\mu)^{\frac{1}{6-\mu}\cdot\frac{3}{2}}S_{H,L}^{\frac{3}{2}}+\O(\varepsilon)\big)
-\frac{t^{2(6-\mu)}}{2(6-\mu)}\big(C(3,\mu)^{\frac{3}{2}}S_{H,L}^{\frac{6-\mu}{2}}-\O(\varepsilon^{\frac{6-\mu}{2}})\big)\\
&\hspace{7mm}-\frac{t^{2\zeta}}{2(6-\mu)}(\O(\varepsilon^{6-\mu-\zeta})-\O(\varepsilon^{\frac{6-\mu}{2}})):=h(t).
\endaligned
$$
Then $h(t)\rightarrow -\infty$ as $t\rightarrow+\infty$, $h(0)=0$ and $h(t)>0$ as $t\rightarrow0^{+}$. In turn,
there exists $t_{\varepsilon}>0$ such that $\sup_{\R^+}h$ is attained at $t_{\varepsilon}$. Differentiating $h$, we obtain
$$
\big(C(3,\mu)^{\frac{3}{2(6-\mu)}}S_{H,L}^{\frac{3}{2}}+\O(\varepsilon)\big)
-t_{\varepsilon}^{2(6-\mu)-2}
\big(C(3,\mu)^{\frac{3}{2}}S_{H,L}^{\frac{6-\mu}{2}}-\O(\varepsilon^{\frac{6-\mu}{2}})\big)
= t_\vr^{2\zeta-2}\big(\O(\varepsilon^{6-\mu-\zeta})-\O(\varepsilon^{\frac{6-\mu}{2}})\big).
$$
Since $0<\mu<3$ and $5-\mu<\zeta<6-\mu$ then
$6-\mu-\zeta<(6-\mu)/2.$
Hence, as $\varepsilon\rightarrow0^{+}$ we have
$$
t_{\varepsilon}<S_{H,L}(\varepsilon):=\left(\frac{C(3,\mu)^{\frac{3}{2(6-\mu)}}S_{H,L}^{\frac{3}{2}}+\O(\varepsilon)}
{C(3,\mu)^{\frac{3}{2}}S_{H,L}^{\frac{6-\mu}{2}}-\O(\varepsilon^{\frac{6-\mu}{2}})}\right)^{\frac{1}{2(6-\mu)-2}}
$$
and there exists $t_{0}>0$ such that, for $\varepsilon>0$ small enough,
$t_{\varepsilon}\geq t_{0}.$
Notice that the function
$$
t\mapsto \frac{t^{2}}{2}\big(C(3,\mu)^{\frac{3}{2(6-\mu)}}S_{H,L}^{\frac{3}{2}}+\O(\varepsilon)\big)
-\frac{t^{2(6-\mu)}}{2(6-\mu)}
\big(C(3,\mu)^{\frac{3}{2}}S_{H,L}^{\frac{6-\mu}{2}}-\O(\varepsilon^{\frac{6-\mu}{2}})\big)
$$
is increasing on $[0,S_{H,L}(\varepsilon)]$, thanks to $t_{0}<t_{\varepsilon}<S_{H,L}(\varepsilon)$, we have
$$\aligned
\max_{t\geq0}\Phi_{1,1,1}(tu_{\varepsilon})
&\leq\frac{5-\mu}{2(6-\mu)}\left(\frac{C(3,\mu)^{\frac{3}{2(6-\mu)}}S_{H,L}^{\frac{3}{2}}+\O(\varepsilon)}
{\Big(C(3,\mu)^{\frac{3}{2}}S_{H,L}^{\frac{6-\mu}{2}}-\O(\varepsilon^{\frac{6-\mu}{2}})\Big)^{\frac{1}{6-\mu}}}
\right)^{\frac{6-\mu}{5-\mu}}\!\!\!- \O(\varepsilon^{6-\mu-\zeta})+\O(\varepsilon^{\frac{6-\mu}{2}}) \\
&\leq\frac{5-\mu}{2(6-\mu)}S_{H,L}^{\frac{6-\mu}{5-\mu}}+\O(\varepsilon)
- \O(\varepsilon^{6-\mu-\zeta})+\O(\varepsilon^{\frac{6-\mu}{2}}).
\endaligned
$$
Since $0<\mu<3$ and $5-\mu<\zeta<6-\mu$, we know that $6-\mu-\zeta<1,$ and
therefore
$$
\max_{t\geq0}\Phi_{1,1,1}(tu_{\varepsilon})
<\frac{5-\mu}{2(6-\mu)}S_{H,L}^{\frac{6-\mu}{5-\mu}},
$$
if $\vr$ is small enough. The proof is completed.
\end{proof}

\noindent
\section{Proof of Theorem \ref{AQ}.}
By Lemma \ref{mountain:1} and Lemma~\ref{EMP1} and the Mountain Pass Theorem without $(PS)$ condition (cf. \cite{MW}), there
exists a $(PS)_{m_{1,1,1}}$-sequence $(u_{n})\subset E$ of $\Phi_{1,1,1}$ with
$$
m_{1,1,1}<\frac{5-\mu}{2(6-\mu)}S_{H,L}^{\frac{6-\mu}{5-\mu}}.
$$
Furthermore, by Lemma \ref{PS1}, there exist $r, \delta>0$ and a sequence $(y_n) \subset \R^3$ such that
$$
\liminf_{n\to\infty}\int_{B_r(y_n)}|u_n|^2 dx\geq\delta.
$$
Since $\Phi_{1,1,1}$ and $\Phi'_{1,1,1}$ are both invariant by translation, without lost of generality we let $y_n=0$ and
\begin{equation} \label{E1}
\int_{B_r(0)}|u_n|^2dx\geq\frac{\delta}{2}.
\end{equation}
Since $(u_n)$ is also bounded, we may assume
$u_n\rightharpoonup u$ in $E$, $u_n(x)\to u(x)$ a.e. in $\R^3$, $u_n\to u$ in
$L^p_{{\rm loc}}(\R^3)$, $p<6$ and $u\not\equiv 0$ by \eqref{E1}. We first check that if $u_{n}\rightharpoonup u$ in $E$, then
\begin{equation}
\label{Conv-1}
\int_{\R^3}\int_{\R^3}\frac{|u_n(y)|^{6-\mu}|u_n(x)|^{4-\mu}u_{n}(x)\varphi(x)}
{|x-y|^{\mu}}dxdy\rightarrow
\int_{\R^3}\int_{\R^3}\frac{|u(y)|^{6-2\mu}|u(x)|^{4-\mu}u(x)\varphi(x)}{|x-y|^{\mu}}dxdy,
\end{equation}
for any $\varphi\in E$, as $n\rightarrow+\infty$.
By the Hardy-Littlewood-Sobolev inequality, we have
$$
|f*|x-y|^{-\mu}|_{{\frac{6}{\mu}}}\leq C |f|_{{\frac{6}{6-\mu}}},\quad \text{for all $f\in {{\frac{6}{6-\mu}}}$}.
$$
Choosing $f_n(y):=|u_n(y)|^{6-\mu}\in L^{\frac{6}{6-\mu}}(\R^3),$ we get
$$
||u_n(y)|^{6-\mu}*|x-y|^{-\mu}|_{{\frac{6}{\mu}}}\leq C |u_n|_{{6}}\leq C.
$$
Therefore, by H\"older inequality with exponents $\frac{5}{5-\mu}$ and $\frac{5}{\mu}$, the sequence
$$
\Big(|u_n(y)|^{6-\mu}*|x-y|^{-\mu}\Big) |u_n(x)|^{4-\mu}u_{n}(x)
$$
is bounded in $L^{6/5}(\R^3)$.
Then, as $n\to+\infty$, by duality we have
$$
\int_{\R^3}\frac{|u_n(y)|^{6-\mu}}
{|x-y|^{\mu}}|u_{n}(x)|^{4-\mu}u_{n}(x)dy\rightharpoonup \int_{\R^3}\frac{|u(y)|^{6-\mu}}
{|x-y|^{\mu}}|u(x)|^{4-\mu}u(x) dy, \quad \mbox{in $L^{\frac{6}{5}}(\R^3)$}
$$
as $n\rightarrow+\infty$. Then \eqref{Conv-1} follows for every $\varphi\in E\subset L^6(\R^3)$.
 For $\varphi \in C^{\infty}_{0}(\mathbb{R}^{3})$,
notice that
$$\aligned
\frac{1}{6-\mu}\int_{\R^3}\int_{\R^3}\frac{G(u(y))g(u(x))\varphi(x)}
{|x-y|^{\mu}}dxdy=\int_{\R^3}\left(\int_{\R^3}\frac{|u(y)|^{6-\mu}+F(u(y))}{|x-y|^{\mu}}dy\right)\Big(|u(x)|^{4-\mu}u(x)+ \frac{1}{6-\mu}f(u(x))\Big)\varphi(x) dx,
\endaligned$$
since $f$ is subcritical in the sense of the Hardy-Littlewood-Sobolev inequality, it is then easy to prove
$$
\int_{\R^3}\int_{\R^3}\frac{G(u_n(y))g(u_n(x))\varphi(x)}
{|x-y|^{\mu}}dxdy\to \int_{\R^3}\int_{\R^3}\frac{G(u(y))g(u(x))\varphi(x)}
{|x-y|^{\mu}}dxdy.
$$
Then $u $ is a nontrivial critical point for $\Phi_{1,1,1}$. By Fatou's Lemma, since $g(s)s-G(s)\geq 0$ for all $s$, we get
$$
\aligned
m_{1,1,1}&\leq\Phi_{1,1,1}(u)-\frac12\langle\Phi'_{1,1,1}(u), u\rangle\\
&\leq \frac{1}{2(6-\mu)}\int_{\R^3}\int_{\R^3}\frac{\big(g(u(x))u(x)-G(u(x))\big)G(u(y))}{|x-y|^{\mu}}dxdy\\
&\leq \lim_{n\to \infty} \frac{1}{2(6-\mu)}\int_{\R^3}\int_{\R^3}\frac{\big(g(u_n(x))u_n(x)-G(u_n(x))\big)G(u_n(y))}{|x-y|^{\mu}}dxdy\\
&=\Phi_{1,1,1}(u_n)-\frac12\langle\Phi'_{1,1,1}(u_n), u_n\rangle
\to m_{1,1,1},
\endaligned
$$
we know
$\Phi_{1,1,1}(u)= m_{1,1,1}$, which means that $u$ is a ground state solution for  $\Phi_{1,1,1}$.
Rewriting the equation \eqref{A} in the form of
$$
-\Delta u+ u
=\left(\int_{\mathbb{R}^3}\frac{H(u(y))u(y)}{|x-y|^{\mu}}dy\right)K(u)\hspace{4.14mm}\mbox{in}\hspace{1.14mm} \R^{3}
$$
where
$$
H(u):=\frac{|u|^{6-\mu}+F(u)}{u},\quad
K(u):=|u|^{4-\mu}u+ \frac{1}{6-\mu}f(u)\in L^{\frac{6}{3-\mu}}(\R^3)+L^{\frac{6}{5-\mu}}(\R^3).
$$
By \cite[Proposition 3.1]{MS2}, we know
$u\in L^p(\R^3)$ for all $p\in [2, 18/(3-\mu))$. Using the growth assumption of $f$ and the
higher integrability of $u$, the Hardy-Littlewood-Sobolev inequality yields, for some $C>0$,
$$
\left|\int_{\R^3}  \frac{G(u(y))}{|x-y|^\mu}dy \right|_{\infty}
\leq C||u|^{6-\mu}+|u|^q+|u|^p|_{\frac{3}{3-\mu}}\leq C\Big(|u|_{\frac{3(6-\mu)}{3-\mu}}^{6-\mu}+
|u|_{\frac{3q}{3-\mu}}^{q}+|u|_{\frac{3p}{3-\mu}}^{p}\Big),
$$
which is finite since the various exponents live within the range $[2, 18/(3-\mu))$. Thus,
$$
-\Delta u+ u
\leq C\Big(|u|^{4-\mu}u+ \frac{1}{6-\mu}f(u)\Big)\hspace{4.14mm}\mbox{in}\hspace{1.14mm} \mathbb{R}^{3}.
$$
By the Moser iteration, the solution $u$ of \eqref{AQ} is classical,
bounded and it decays to zero at infinity.
\qed

\begin{lem}\label{EDP}
There are $C$, $\beta>0$ such that the ground state solution satisfies
$|u(x)|\leq C e^{-\beta|x|}$ for $x\in \mathbb{R}^{3}$.
\end{lem}
\begin{proof}
By the previous discussion, we have
$$
-\Delta u+ \frac{1}{2}u
\leq C\big(|u|^{4-\mu}u+f(u)\big)-\frac{1}{2}u.
$$
Since $u(x)\to 0$  uniformly
as $|x|\rightarrow +\infty,$ we find $\rho_{0} >0$ such that for $ |x|\geq \rho_{0}$ the right hand side is negative.
It is then well known that $-\Delta u+u/2\leq 0$ yields an exponential decay on $\R^3$.
\end{proof}

The following is a comparison result for the mountain pass values with different parameters $\kappa,\nu, \tau>0$, useful in proving the existence result for
\eqref{SCP} when $\vr$ is small enough.

\begin{lem}[Monotonicity of energy levels]
	\label{CO}
 Let $\kappa_i, \nu_i,\tau_i>0$, $i=1,2$, with $\min\{\kappa_2-\kappa_1, \nu_1-\nu_2,\tau_1-\tau_2\}\geq0$. Then
 $$
 m_{\kappa_1,\nu_1,\tau_1}\leq m_{\kappa_2,\nu_2, \tau_2}.
$$
 If additionally, $\max\{\kappa_2-\kappa_1, \nu_1-\nu_2, \tau_1-\tau_2\}>0$, then the inequality is strict.
\end{lem}
\begin{proof}
From Theorem \ref{AQ}, let $u$ be a weak solution of problem \eqref{A} with coefficients $\kappa_2,\nu_2, \tau_2$ at the energy level $\Phi_{\kappa_2,\nu_2,\tau_2}(u)=m_{\kappa_2,\nu_2,\tau_2}$.
By $(f_3)$, we know there is a unique $t=t(u)>0$ such that
 $$
 \Phi_{\kappa_2,\nu_2,\tau_2}(t(u)u)=\max_{s\geq0}\Phi_{\kappa_2,\nu_2,\tau_2}(su),\qquad
  t(u)u\in {\cal N}_{\kappa_2,\nu_2,\tau_2}.
 $$
Since $u\in {\cal N}_{\kappa_2,\nu_2,\tau_2}$, we know $t(u)=1$ and so
$$
\Phi_{\kappa_2,\nu_2,\tau_2}(u)=\max_{t\geq0}\Phi_{\kappa_2,\nu_2,\tau_2}(tu).
$$
Similarly, there exists $t_0>0$ such that $
\Phi_{\kappa_1,\nu_1,\tau_1}(t_0u)=\max_{t\geq 0}
\Phi_{\kappa_1,\nu_1,\tau_1}(tu)$. Then
$$\aligned
m_{\kappa_1,\nu_1,\tau_1}&=\inf_{w\in E\backslash\{0\}} \max_{t\geq 0}
\Phi_{\kappa_1,\nu_1,\tau_1}(tw)
\leq \max_{t\geq 0}
\Phi_{\kappa_1,\nu_1,\tau_1}(tu)\\
&=\Phi_{\kappa_1,\nu_1, \tau_1}(t_0u)
\leq \Phi_{\kappa_2,\nu_2,\tau_2}(t_0u)
 \leq \Phi_{\kappa_2,\nu_2,\tau_2}(u)=m_{\kappa_2,\nu_2,\tau_2},
\endaligned
$$
which concludes the proof. The proof of the strict inequality is similar.
\end{proof}

\section{Critical equation with nonlinear potential}
In this section we will consider the existence and concentration of the solutions of equation \eqref{SCP1}. Consider
\begin{equation}
\label{SCC1}
-\Delta u + u  =\frac{1}{6-\mu}\Big(\int_{\R^3}  \frac{Q(\vr y)G( u(y))}{|x-y|^\mu}dy\Big)Q(\vr x)g(u) \,\,\,\, \mbox{in $\R^{3}$},
\tag{SCC1}
\end{equation}
where we still use the notions $G(u)=|u|^{6-\mu}+F(u)$.
By changing variable, it is possible to see that the above equation is equivalent to equation \eqref{SCP1}. The energy functional associated to $(SCC1)$ is
$$
I_{\vr}(u):=\frac12\int_{\mathbb{R}^3} (|\nabla u|^2+ |u|^2)dx-\tilde{\Psi}(u),
\qquad
\tilde{\Psi}(u):=\frac{1}{2(6-\mu)}\int_{\mathbb{R}^3}\int_{\mathbb{R}^3}   \frac{Q(\vr y)G(u(y))Q(\vr x)G(u(x))}{|x-y|^\mu}dxdy.
$$
The Nehari manifold associated to $I_{\vr}$ will be denote by ${\cal{N}_\vr}$, that is,
${\cal{N}_\vr}=\big\{u\in E:u\neq0, \langle I'_{\vr}(u), u\rangle=0\big\}$
and  there exists $\alpha>0$, independent of $\vr$, such that
\begin{equation*}
\|u\|\geq \alpha, \,\,\,\,\,\,\, \forall u\in \cal{N}_\vr .
\end{equation*}
Similar to Lemma \ref{mountain:1}, we know $I_{\vr}$ also satisfies the Mountain Pass Geometry and assumption $(f_3)$ implies that the least energy can be characterized by
\begin{equation} \label{R0}
c_{\vr}=\inf_{u\in\cal{N}_\vr}I_{\vr}(u)=\inf_{u\in E\backslash\{0\}} \max_{t\geq 0}
I_{\vr}(tu),
\end{equation}
and there exists $c>0$, which is independent of $\vr$,  such that $c_{\vr}>c$.

\subsection{Truncating techniques}

For $ d\in [\nu_{\min}, \nu_{\max}]$, we set
$$
Q^d(\vr x):=\min\{d, Q(\vr x)\}
$$
and introduce the first auxiliary problem for equation $(SCC1)$ by considering
$$
-\Delta u + u  =\frac{1}{6-\mu}\Big(\int_{\R^3}  \frac{Q^d(\vr y)G( u(y))}{|x-y|^\mu}dy\Big)Q^d(\vr x)g(u).
$$
The associated energy functional is defined by
$$
\begin{array}{ll}
\displaystyle I^{d}_{\vr}(u)=\frac12\int_{\R^3}\big(|\nabla u|^2+ |u|^2\big)dx-\frac{1}{2(6-\mu)}\int_{\mathbb{R}^3}\int_{\mathbb{R}^3}   \frac{Q^d(\vr y)G(u(y))Q^d(\vr  x)G(u(x))}{|x-y|^\mu}dxdy.
\end{array}
$$
The associated Nehari manifold is ${\cal{N}}^{d}_{\vr}=\big\{u\in E:u\neq0,\langle {(I^{d}_{\vr}})'(u),u\rangle=0\big\}$
and the least energy is $c^{d}_{\vr}$.

\begin{lem}\label{AFE}
Suppose that $f$ satisfies $(f_1)$-$(f_3)$. Then
$$
\limsup_{\vr\to0}c^{d}_{\vr}\leq m_{1, Q^d(0), Q^d(0)}.
$$
\end{lem}
\begin{proof}
 Let $u$ be a ground state solution of ~\eqref{A} with coefficients $(1,Q^d(0), Q^d(0))$, that is
 $$
 \Phi_{1,Q^d(0), Q^d(0)}(u)=m_{1, Q^d(0), Q^d(0)}.
 $$
Then there exists a unique $t_\vr=t_\vr(u)>0$ such that $t_\vr u\in {\cal{N}}^{d}_{\vr}$ and $c^{d}_{\vr}\leq I^{d}_{\vr}(t_\vr u).$
From the boundedness of $Q$, by the arguments in Lemma \ref{mountain:1}, there exists
$T>0$ independent of $\vr$ with $I^{d}_{\vr}(su)<0$ for all $s\geq T$. Consequently,
$t_\vr<T$ and we may assume that $t_\vr\to t_0$.
Observe that
$$
\begin{array}{ll}
\displaystyle I^{d}_{\vr}(t_\vr u)
=\displaystyle\Phi_{1, Q^d(0), Q^d(0)}(t_\vr u)-\frac{1}{2(6-\mu)}\int_{\mathbb{R}^3}\int_{\mathbb{R}^3}   \frac{[Q^d(\vr y)Q^d(\vr x)-Q^d(0)Q^d(0)]G(u(y))G(u(x))}{|x-y|^\mu}dxdy.
\end{array}
$$
Once that $Q$ is bounded and $t_\vr \to t_0$, applying the Lebesgue's Dominated Convergence theorem, we know
\begin{align*}
\limsup\limits_{\vr\to 0} c^{d}_{\vr} &\leq \limsup_{\vr\to 0} I^{d}_{\vr}(t_\vr u)  \\
&=\limsup\limits_{\vr\to 0}  \big(\Phi_{1,Q^d(0), Q^d(0)}(t_\vr u)+o_\vr(1)\big)  \\
& =\Phi_{1,Q^d(0), Q^d(0)}(t_0 u)
 \leq \Phi_{1,Q^d(0), Q^d(0)}(u)
 =m_{1,Q^d(0),Q^d(0)},
\end{align*}
which concludes the proof.
\end{proof}

\noindent
Next, we prove an upper bound for the Mountain Pass level $c_{\vr}$ in \eqref{R0}.
\begin{lem}\label{UB1}
There holds
 $$
 \limsup_{\vr\to0}c_{\vr}\leq  m_{1,\nu_{\max},\nu_{\max}}.
 $$
\end{lem}
\begin{proof}
If $d=\nu_{\max}$, then $Q^{d}(\vr x)=Q(\vr x)$. Consequently, $c^{d}_{\vr}=c_{\vr}.$
Then, by Lemma \ref{AFE},
$$
\limsup_{\vr\to0}c_{\vr}\leq   m_{1,\nu_{\max}, \nu_{\max}}.
$$
This completes the proof.
\end{proof}


\noindent
To consider the existence of solutions concentrating at the nonlinear potential, we will partially truncate the nonlinear potential $Q$ in front of the subcritical term and introduce the second auxiliary problem for equation $(SCC1)$. For $ e\in [\nu_{\min}, \nu_{\max})$, we set $Q^e(\vr x):=\min\{e, Q(\vr x)\} $
and consider
$$
-\Delta u + u  =\Big(\int_{\R^3}  \frac{Q(\vr y)|u(y)|^{6-\mu}+Q^e(\vr y)F(u(y))}{|x-y|^\mu}dy\Big)[Q(\vr x)|u|^{4-\mu}u+\frac{Q^e(\vr x)}{6-\mu}f(u)] \quad \mbox{in} \quad \mathbb{R}^3.
$$
The associated energy functional is defined by
$$
\begin{array}{ll}
\displaystyle \tilde{I}^{e}_{\vr}(u)=\frac12\int_{\R^3}\big(|\nabla u|^2+ |u|^2\big)dx-\frac{1}{2(6-\mu)}\int_{\mathbb{R}^3}\int_{\mathbb{R}^3}   \frac{[Q(\vr y)|u(y)|^{6-\mu}+Q^e(\vr y)F(u(y))][Q(\vr x)|u(x)|^{6-\mu}+Q^e(\vr x)F(u(x))]}{|x-y|^\mu}dxdy,
\end{array}
$$
the corresponding Nehari manifold and least energy are ${\tilde{\cal{N}}}^{e}_{\vr}$
and $\tilde{c}^{e}_{\vr}$. Related to the above functional, we have an important lower bound for the level $\tilde{c}^{e}_{\vr}$.
\begin{lem}\label{PCO}
$\tilde{c}^{e}_{\vr}\geq {m}_{1, \nu_{\max}, e}.$
\end{lem}
\begin{proof}
 Since
$Q^e(\vr x)\leq e$ and $Q(\vr x)\leq \nu_{\max},$
from the characterization of the value ${m}_{1, \nu_{\max}, e}$,  we know
$$
\inf_{u\in E}\max_{t\geq0}\tilde{I}^{e}_{\vr}(tu)\geq \inf_{u\in E}\max_{t\geq0}\Phi_{1, \nu_{\max}, e}(tu),
$$
namely the assertion.
\end{proof}

\subsection{Existence for Theorem \ref{Result1}}
In this subsection, we  will prove Theorem \ref{Result1}.

\begin{lem}\label{Existence}
Suppose that the potential function $Q$ satisfies $(Q)$ and the nonlinearity $f$ satisfies $(f_1)$-$(f_3)$. Then the minimax value $c_{\vr}$ is achieved if $\vr$ is small enough. Hence, problem \eqref{SCC1} admits a least energy solution if  $\vr$ is small enough.
\end{lem}
\begin{proof}
From Lemma \ref{UB1}, there holds
\begin{equation*}
 \limsup_{\vr\to0}c_{\vr}\leq  m_{1,\nu_{\max}, \nu_{\max}}.
\end{equation*}
Furthermore, we know
\begin{equation*}
 m_{1,\nu_{\max}, \nu_{\max}}<\frac{5-\mu}{2(6-\mu)}\nu_{\max}^{-\frac{2}{5-\mu}}S_{H,L}^{\frac{6-\mu}{5-\mu}}.
\end{equation*}
Since the least energy $c_{\vr}$ can be characterized by
$$
c_{\vr}=\inf_{u\in{\cal{N}}_{\vr}}I_{\vr}(u),
$$
we can choose a minimizing sequence $(u_{n}) \subset {\cal{N}}_{\vr}$ of $I_{\vr}$ such that $I_{\vr}(u_{n})\to c_{\vr}$. By Ekeland's variational principle \cite{MW}, we may also assume it is a bounded $(PS)$ sequence at $c_{\vr}$. Without loss of generality, we assume that $u_n\rightharpoonup u_\vr$ in $E$ with $I'_{\vr}(u_\vr)=0$. To complete the proof, we need to show that $u_{\vr}\neq0$ if $\vr$ is small enough. On the contrary  we assume that there exists a sequence $\vr_j\to0$ with $u_{{\vr}_j}=0$. For each fixed $j$,
let $(u_{n}) \subset {\cal{N}}_{\vr_j}$ be a $(PS)$ sequence of $I_{\vr}$ at $c_{{\vr}_j}$
such that $u_n\rightharpoonup u_{{\vr}_j}=0$ in  $E$. Select $\nu_{\min}\leq e<\nu_{\max}$ and consider the functional $\tilde{I}^{e}_{\vr_j}$.
Note that for each $u_{n}$ there is a unique $t_{n}$ such that $t_{n}u_{n}\in{\tilde{\cal{N}}}^{e}_{\vr_j}$,  we claim that the sequence $(t_{n})$ is bounded. Indeed, suppose by contradiction that  $t_{n}\to \infty$ as $n\to \infty$. Since $(u_{n})$ is bounded and $\|u_{n}\|^2\geq\alpha$, we know that there exist $(y_{n}) \subset \R^{3}$ and $r,\delta >0$ such that
$$
\int_{B_{r}(y_{n})}|u_{n}|^{2}dx \geq \delta, \quad n \in \mathbb{N}.
$$
Otherwise, $u_{n}\to 0$ in
 $L^s(\R^3)$,  $2<s< 6$,
and we can get
$$
\begin{array}{ll}
\displaystyle I_{ \vr_j}(u_{n})-\frac{1}{2}\langle {I'_{\vr_j}}(u_{n}), u_{n} \rangle=\frac{5-\mu}{2(6-\mu)}\int_{\mathbb{R}^3}\int_{\mathbb{R}^3}   \frac{Q(\vr_j y)|u_n(y)|^{6-\mu}Q(\vr_j x)|u_n(x)|^{6-\mu}}{|x-y|^\mu}dxdy+o_n(1).
\end{array}
$$
Notice that $(u_{n}) \subset {\cal{N}}_{\vr_j}$  is bounded minimizing sequence at $c_{\vr_j}$, we have
\begin{equation}\label{E12}
\begin{array}{ll}
\displaystyle\frac12\|u_{n}\|^2-\frac{1}{2(6-\mu)}\int_{\mathbb{R}^3}\int_{\mathbb{R}^3}   \frac{Q(\vr_j y)|u_n(y)|^{6-\mu}Q(\vr_j x)|u_n(x)|^{6-\mu}}{|x-y|^\mu}dxdy\to c_{\vr_j}.
\end{array}
\end{equation}
\begin{equation}\label{E13}
\|u_{n}\|^{2}=\int_{\mathbb{R}^3}\int_{\mathbb{R}^3}   \frac{Q(\vr_j y)|u_n(y)|^{6-\mu}Q(\vr_j x)|u_n(x)|^{6-\mu}}{|x-y|^\mu}dxdy.
\end{equation}
And so we get
\begin{equation}\label{E14}
\|u_{n}\|^{2}\leq \nu^2_{\max}\int_{\mathbb{R}^3}
\int_{\mathbb{R}^3}\frac{|u_{n}(x)|^{6-\mu}|u_{n}(y)|^{6-\mu}}
{|x-y|^{\mu}}dxdy
\leq \nu^2_{\max} S_{H,L}^{\mu-6}\|u_{n}\|^{2(6-\mu)}.
\end{equation}
If $\|u_{n}\|\rightarrow0$, then $c_{\vr_j}=0$, a contradiction. Consequently, $\|u_{n}\|\nrightarrow0$. So by \eqref{E14} we get $\|u_{n}\|\geq \nu_{\max}^{-\frac{1}{5-\mu}} S_{H,L}^{\frac{6-\mu}{2(5-\mu)}}$. Then from \eqref{E12}, \eqref{E13} and \eqref{E14} we easily conclude that $$c_{\vr_j}\geq \frac{5-\mu}{2(6-\mu)}\nu_{\max}^{-\frac{2}{5-\mu}}S_{H,L}^{\frac{6-\mu}{5-\mu}},$$ which contradicts with the assumption. Hence $(u_{n})$ is non-vanishing. Thus, $v_{n}(x)=u_{n}(x+y_n)$ is bounded in $E$ and its weak limit $v \in  E$ is not zero, namely $v \not=0$. Hence, there is $\Omega \subset \R^3$ with $|\Omega|>0$ such that $v(x) >0$ for all $x \in \Omega.$
Since $(u_{n})$ and $V$ are bounded and $\displaystyle \inf_{x \in \R^3}Q(\vr x)>0$,
$$
t_{n}^2\|v_{n}\|^2\geq {t_{n}^{2\zeta}}C
\int_{\R^3}
\int_{\R^3}\frac{|v_{n}(x)|^{\zeta}|v_{n}(y)|^{\zeta}}
{|x-y|^{\mu}}dxdy,
$$
which implies that $(t_{n})$ is bounded. In what follows
we assume that $t_{n}\to t_{0}>0$ as $n\to \infty$. Hence,
\begin{align*}
\tilde{c}^{e}_{\vr_j} &\leq \tilde{I}^{e}_{\vr_j}(t_{n}u_{n})=I_{\vr_j}(t_{n}u_{n})+\frac{1}{(6-\mu)}\int_{\mathbb{R}^3}\int_{\mathbb{R}^3}   \frac{Q(\vr_j  y)|u_n(y)|^{6-\mu}[Q(\vr_j  x)-Q^e(\vr_j x)]F(u_n(x))}{|x-y|^\mu}dxdy \\
&+
\frac{1}{2(6-\mu)}\int_{\mathbb{R}^3}\int_{\mathbb{R}^3}   \frac{[Q(\vr_j  y)Q(\vr_j  x)-Q^e(\vr_j  y)Q^e(\vr_j  x)]F(u_n(x))F(u_n(y))}{|x-y|^\mu}dxdy.
 \end{align*}
Notice that $u_{n}\rightharpoonup u_{\vr_j }=0$ in $E$ and $u_{n} \to 0$ in $L^q_{{\rm loc}}(\R^3)$ for $q$ subcritical.
Choose $t$ and $s$ close to $\frac{6}{6-\mu}$ with
$$
1/t+\mu/3+1/s=2
$$
and apply the Hardy-Littlewood-Sobolev inequality, we know
$$
\Big|\int_{\mathbb{R}^3}
\int_{\mathbb{R}^3}\frac{Q(\vr_j  y)|u_n(y)|^{6-\mu}[Q(\vr_j  x)-Q^e(\vr_j  x)]F(u_n(x))}
{|x-y|^{\mu}}dxdy\Big|\leq C|u_n|^{6-\mu}_{t(6-\mu)}|[Q(\vr_j  x)-Q^e(\vr_j  x)]F(u_n)|_{s}.
$$
Observe that
$$
\int_{\R^3}|[Q(\vr_j x)-Q^e(\vr_j  x)]F(u_n)|^sdx=\int_{\{x: Q(\vr_j x)\geq e\}}|[Q(\vr_j  x)-e]F(u_n)|^sdx=o_n(1),
$$
since ${\{x: Q(\vr_j  x)\geq e\}}$ is bounded and $f(s)$ is of subcritical growth, we know
$$
\int_{\mathbb{R}^3}
\int_{\mathbb{R}^3}\frac{Q(\vr_j  y)|u_n(y)|^{6-\mu}[Q(\vr_j  x)-Q^e(\vr_j  x)]F(u_n(x))}
{|x-y|^{\mu}}dxdy=o_n(1).
$$
Similarly,
$$
 \int_{\mathbb{R}^3}\int_{\mathbb{R}^3}   \frac{[Q(\vr_j  y)Q(\vr_j  x)-Q^e(\vr_j  y)Q^e(\vr_j  x)]F(u_n(x))F(u_n(y))}{|x-y|^\mu}dxdy=o_n(1).
$$
From the above arguments and the fact that $I_{\vr_j}(t_{n}u_{n})\leq I_{\vr_j}(u_{n})$, since $(u_{n}) \subset {\cal{N}}_{\vr_j}$, we know
 $$
\begin{array}{ll}
\displaystyle\tilde{c}^{e}_{\vr_j}\leq I_{\vr_j}(t_{n}u_{n})+o_n(1)\leq I_{\vr_j}(u_{n})+o_n(1).
\end{array}
$$
Hence $c^{e}_{\vr_j}\leq c_{\vr_j} $ as $n\to\infty$. From Lemma \ref{PCO}, since there holds
$\tilde{c}^{e}_{\vr}\geq {m}_{1, \nu_{\max}, e}$,
we know ${m}_{1, \nu_{\max}, e}\leq c_{\vr_j}.$
Taking the limit $j \to +\infty$ and using Lemma \ref{UB1}, we get
${m}_{1, \nu_{\max}, e}\leq m_{1,\nu_{\max}, \nu_{\max}},$
applying Lemma \ref{CO} with the fact that $e<\nu_{\max}$, this yields a contradiction that is $u_\vr\neq 0$. Then, repeat the arguments in section 3, we know $I_\vr(u_\vr)=c_\vr$ which finishes the proof.
\end{proof}

\noindent
The following Brezis-Lieb type lemma, here specialized for $N=3$, for the nonlocal term is proved in \cite{GY}.

\begin{lem} \label{BLN}
	Let $0<\mu<3$. If $(u_{n})$ is a bounded sequence in $L^{6}(\mathbb{R}^3)$ such that $u_{n}\rightarrow u$ almost everywhere in $\mathbb{R}^3$,
	then the following hold,
$$
\int_{\mathbb{R}^3}(|x|^{-\mu}\ast |u_{n}|^{6-\mu})|u_{n}|^{6-\mu}dx-\int_{\mathbb{R}^3}(|x|^{-\mu}\ast |u_{n}-u|^{6-\mu})|u_{n}-u|^{6-\mu}dx\rightarrow\int_{\mathbb{R}^3}(|x|^{-\mu}\ast |u|^{6-\mu})|u|^{6-\mu}dx,
$$
as $n\rightarrow\infty$.
\end{lem}
\begin{lem}\label{CP}
Let $(u_{n})$ be the sequence of solutions obtained in Lemma \ref{Existence} with parameter $\vr_{n}\to0$. Then, there is $y_{n}\in\R^3$ such that
$$
\lim_{n\to\infty}\dist ({\vr_n}y_{n},  \mathcal{Q})=0,
$$
such that the sequence $v_{n}(x):=u_{n}(x+y_{n})$ converges strongly in $E$ to a ground state solution $v$ of
    $$
-\Delta v +v  =\nu^2_{\max}\left(\int_{\mathbb{R}^3}\frac{|v(y)|^{6-\mu}+F(v(y))}{|x-y|^{\mu}}dy\right)\Big(|v|^{4-\mu}v+ \frac{1}{6-\mu}f(v)\Big) \quad \mbox{in $\mathbb{R}^{3}$}.
$$
\end{lem}
\begin{proof}
Let $(u_{n})$ be the  sequence of solutions  obtained in Lemma \ref{Existence} with parameter $\vr_{n}\to0$. It is easy to see that $(u_{n})$ is bounded in $E$. Moreover, repeat the arguments in Lemma \ref{Existence}, we know that there exist $r, \delta>0$ and a sequence $(y_{n}) \subset \R^3$ such that
\begin{equation}\label{J0}
\liminf_{n\to\infty}\int_{B_r(y_{n})}|u_{n}|^2dx\geq\delta.
\end{equation}
Setting $v_{n}(x):=u_{n}(x+y_{n})$ and $\widetilde{Q}_{\vr_n}(x):=Q(\vr_n (x+y_{n}))$, we see that $v_{n}$ solves problems
\begin{equation*}
-\Delta v+v =\frac{1}{6-\mu}\Big(\int_{\R^3}  \frac{\widetilde{Q}_{\vr_n}(y)G( v(y))}{|x-y|^\mu}dy\Big)\widetilde{Q}_{\vr_n}(x)g(v(x)) \quad \mbox{in $\R^{3}$}.
\end{equation*}
We shall use $\tilde{I}_{\vr_n}$ to denote the corresponding energy functional. Since $v_{n}(x):=u_{n}(x+y_{n})$ is also bounded, from  \eqref{J0}, we may assume that $v_{n}\rightharpoonup v$ in $E$ with $v\neq0$ and $v \geq 0$.
The sequence $(\vr_ny_{n})$ must be bounded and up to sequence $\vr_ny_{n}\to y_{0}\in \mathcal{Q}$. Argue by contradiction, we assume that $\vr_ny_{n}\to\infty$, as $n\to\infty$, we may suppose that $Q(\vr_n y_{n})\to Q_0< \nu_{\max}$. Since $\langle\tilde{I}'_{\vr_n}(v_{n}),\varphi\rangle=0$
for any $\varphi\in C^{\infty}_0(\R^3)$, equivalently, we have
\begin{equation}\label{Lim1}
\begin{array}{ll}
\displaystyle\int_{\R^3} (\nabla v_{n}\nabla \varphi+ v_{n}\varphi)dx-\frac{1}{6-\mu}\int_{\R^3}\Big(\int_{\R^3}  \frac{\widetilde{Q}_{\vr_n}(y)G( v_n(y))}{|x-y|^\mu}dy\Big)\widetilde{Q}_{\vr_n}(x)g(v_n(x))\varphi(x) dx=0.
\end{array}
\end{equation}
From the regularity arguments in the section 3, we know $$
\left|\int_{\R^3}  \frac{G(v_n(y))}{|x-y|^\mu}dy \right|_{\infty}
\leq C\Big(|v_n|_{\frac{3(6-\mu)}{3-\mu}}^{6-\mu}+
|v_n|_{\frac{3q}{3-\mu}}^{q}+|v_n|_{\frac{3p}{3-\mu}}^{p}\Big),
$$
thus
$$
C=\sup_{n\in\N}\sup_{x\in\R^3}\Big|\int_{\R^3}  \frac{\widetilde{Q}_{\vr_n}(y)G( v_n(y))}{|x-y|^\mu}dy\Big|<\infty.
$$
Whence, we have
$$
\begin{array}{ll}
\displaystyle\Big|\int_{\R^3}\Big(\int_{\R^3}  \frac{\widetilde{Q}_{\vr_n}(y)G(v_n(y))}{|x-y|^\mu}dy\Big)\Big(\widetilde{Q}_{\vr_n}(x)g(v_n(x))-Q_0g(v(x))\Big)\varphi(x) dx\Big|
\displaystyle\leq C\Big|\int_{\R^3}\Big(\widetilde{Q}_{\vr_n}(x)g(v_n(x))-Q_0g(v(x))\Big)\varphi(x) dx\Big|.
\end{array}
 $$
Then, since
$$
\int_{\R^3}\Big(\widetilde{Q}_{\vr_n}(x)g(v_n(x))-Q_0g(v(x))\Big)\varphi(x)dx\to 0,\quad  \forall \varphi \in C^{\infty}_0(\R^3),
$$
we have
\begin{equation*}
\Big|\int_{\R^3}\Big(\int_{\R^3}  \frac{\widetilde{Q}_{\vr_n}(y)G(v_n(y))}{|x-y|^\mu}dy\Big)\Big(\widetilde{Q}_{\vr_n}(x)g(v_n(x))-Q_0g(v(x))\Big)\varphi(x) dx\Big|\to 0, \quad  \forall \varphi \in C^{\infty}_0(\R^3).
\end{equation*}
Repeating the arguments in the proof of Theorem \ref{AQ}, we obtain
$$
\begin{array}{ll}
\displaystyle\Big|\int_{\R^3}\Big(\int_{\R^3}  \frac{\widetilde{Q}_{\vr_n}(y)G( v_n(y))-Q_0G(v(y))}{|x-y|^\mu}dy\Big)g(v(x))\varphi(x) dx\Big|\to 0,\quad   \forall \varphi \in C^{\infty}_0(\R^3).
\end{array}
 $$
Taking the  limit in equation \eqref{Lim1}, we get
 that $v$ is nothing but a solution of the equation
 $$
-\Delta v +v  =Q_0^2\left(\int_{\mathbb{R}^3}\frac{|v(y)|^{6-\mu}+F(v(y))}{|x-y|^{\mu}}dy\right)\Big(|v|^{4-\mu}v+ \frac{1}{6-\mu}f(v)\Big) \quad \mbox{in} \quad \mathbb{R}^{3}.
$$
Observe that
$I_{\vr_n}(u_{n})=\tilde{I}_{\vr_n}(v_{n}),$
and by  Fatou's Lemma and Lemma \ref{CO}, we can get
$$
\aligned
& m_{1, \nu_{\max},\nu_{\max} }<m_{1, Q_0, Q_0}
\leq\Phi_{1, Q_0, Q_0}(v)\\
&= \Phi_{1, Q_0, Q_0}(v)-\frac{1}{\theta(6-\mu)}\langle \Phi'_{1, Q_0, Q_0}(v), v\rangle\\
&=\Big(\frac12-\frac{1}{\theta(6-\mu)}\Big)\|v\|^2+\frac{Q^2_0}{2\theta(6-\mu)} \int_{\mathbb{R}^3}\int_{\mathbb{R}^3}\frac{G(v(y))[2g(v(x))v(x)-\theta G(v(x))]}{|x-y|^{\mu}} dxdy\\
&\leq\liminf_{n\to\infty}\Big\{\Big(\frac12-\frac{1}{\theta(6-\mu)}\Big)\|v_n\|^2+\frac{1}{2\theta(6-\mu)} \int_{\mathbb{R}^3}\int_{\mathbb{R}^3}\frac{{Q}_{\vr_n}(y)G(v_n(y)){Q}_{\vr_n}(x)[2g(v_n(x))v_n(x)-\theta G(v_n(x))]}{|x-y|^{\mu}}dxdy\Big\}\\
&=\liminf_{n\to\infty}\Big\{\tilde{I}_{\vr_n}(v_{n})-\frac{1}{\theta(6-\mu)}\langle \tilde{I}'_{\vr_n}(v_n), v_n\rangle\Big\}
=\liminf_{n\to\infty}c_{\vr_n}.
 \endaligned
 $$
This contradicts to Lemma \ref{UB1} which says
 $$
 \limsup_{n\to\infty}c_{\vr_n}\leq  m_{1, \nu_{\max},\nu_{\max} }.
 $$
Thus $(\vr_ny_{n})$ is bounded and we may assume that $\vr_ny_{n}\to y_{0}$.
Next we are going to prove $
y_0\in \mathcal{Q}.
$ If $
y_0\notin \mathcal{Q}$, by the definitions of $\mathcal{Q}$,  then it is easy to see $$ m_{1, \nu_{\max},\nu_{\max} }<m_{1, Q(y_0), Q(y_0)}.$$
Let $v$ be the weak limit of the sequence $v_{n}(x):=u_{n}(x+y_{n})$ then
 $v$ satisfies
\begin{equation}\label{Lim3}
-\Delta v +v  =Q(y_0)^2\left(\int_{\mathbb{R}^3}\frac{|v(y)|^{6-\mu}+F(v(y))}{|x-y|^{\mu}}dy\right)\Big(|v|^{4-\mu}v+ \frac{1}{6-\mu}f(v)\Big)
\end{equation}
and
$$
\begin{array}{ll}
\displaystyle  m_{1, \nu_{\max},\nu_{\max} }<m_{1, Q(y_0), Q(y_0)}\leq\liminf_{n\to\infty}c_{\vr_n}.
 \end{array}
 $$
which contradicts Lemma \ref{UB1}, since
 $$
 \limsup_{n\to\infty}c_{\vr_n}\leq m_{1, \nu_{\max},\nu_{\max} }.
 $$
Therefore $y_0\in \mathcal{Q}$, which means $\dist (\vr_n y_n, \mathcal{Q})\to 0$.
By repeating the arguments in Lemma \ref{AFE}, we get
$$
\lim_{n\to\infty}\tilde{I}_{\vr_n}(v_{n})\leq m_{1, Q(y_0), Q(y_0)}=m_{1, \nu_{\max}, \nu_{\max}},
$$
consequently,
$$
\Phi_{1,\nu_{\max}, \nu_{\max}}(v)=\Phi_{1,Q(y_0),Q(y_0)}(v)=m_{1, \nu_{\max}, \nu_{\max}},
$$
 and so
  $v$ in fact is  a ground state solution of the equation \eqref{Lim3} with $Q(y_0)=\nu_{\max}$.
Finally we show that $(v_{n})$ converges strongly to $v$ in $E$. Since $Q$ is uniformly continuous, using Lemma \ref{BLN},
$$
\tilde{I}_{\vr_n}(v_{n}-v)=\tilde{I}_{\vr_n}(v_{n})-\Phi_{1,\nu_{\max}, \nu_{\max}}(v)+o_n(1).
$$
Since
$$
\lim_{n\to\infty}\tilde{I}_{\vr_n}(v_{n})=\Phi_{1,\nu_{\max}, \nu_{\max}}(v),
$$
it follows that $\tilde{I}_{\vr_n}(v_{n}-v)\to 0.$
Similarly, $\tilde{I}'_{\vr_n}(v_{n}-v)\to0,$
which implies
$$
\lim_{n\to\infty}\langle\tilde{I}'_{\vr_n}(v_{n}-v), v_{n}-v\rangle=0.
$$
Hence,
$$
\|v_{n}-v\|^2\leq C\lim_{n\to\infty}\Big(\tilde{I}_{\vr_n}(v_{n}-v)-\frac{1}{\theta(6-\mu)}\langle\tilde{I}'_{\vr_n}(v_{n}-v),v_{n}-v\rangle\Big)=0,
$$
showing that $v_{n}\to v$ in $E$. This ends the proof.
\end{proof}
\begin{lem}\label{0AI}
Let  $(v_{n})$ be the sequence obtained in Lemma \ref{CP}. Then, there exists $C>0$ independent of $n$ such that $|v_{n}|_{{\infty}}\leq C$ and
$v_{n}(x)\to 0$ as $|x|\to\infty$, uniformly in $n\in \mathbb{N}.$
Furthermore there are $C$, $\beta>0$ with
$$
|v_{n}(x)|\leq C \exp(-\beta|x|),\quad \forall x\in \R^3.
$$
\end{lem}
\begin{proof}
From Lemma \ref{CP} we know that ${\vr_n} y_{n}\to y_0\in \mathcal{Q}$
as $n\to\infty$ and the sequence $v_{n}(x):=u_{n}(x+y_{n})$ converges strongly to
a solution $v$ of the equation
$$
-\Delta v +v  =Q(y_0)^2\left(\int_{\mathbb{R}^3}\frac{|v(y)|^{6-\mu}+F(v(y))}{|x-y|^{\mu}}dy\right)\Big(|v|^{4-\mu}v+ \frac{1}{6-\mu}f(v)\Big) \quad \mbox{in} \quad \mathbb{R}^{3}.
$$
From the regularity arguments in section 2,
$$
\sup_{n\in\N}\sup_{x\in\R^3}\Big|\int_{\R^3}  \frac{\widetilde{Q}_{\vr_n}(y)G( v_n(y))}{|x-y|^\mu}dy\Big|<\infty,
$$
and $v_{n}\in L^{q}(\R^3)$ for all $2\leq q<\infty$. Furthermore, the elliptic regularity theory implies that $v_{n}\in C^{2}(\R^3)$ and
\begin{equation*}
 -\Delta v_{n} \leq h(v_{n})\quad \mbox{in} \quad \mathbb{R}^3,
\end{equation*}
where $h(v_{n})\in  L^{t}(\R^3)$, $t>\frac32$. Then, we learn that $|v_{n}|_{\infty}\leq C$ and
$$
\lim\limits_{\mid x\mid\rightarrow\infty}{v_{n}}(x)=0 \quad \mbox{uniformly in $n \in \mathbb{N}$}.
$$
Recall that by (\ref{J0}),
$$
 \delta \leq \int_{B_{r}(y_n)}|u_n|^{2}dx,
$$
then we obtain
$$
\delta\leq\int_{B_{r}(0)}|v_n|^{2}dx \leq |B_r||v_n|^2_{\infty},
$$
from where it follows
$|v_n|_{\infty} \geq \delta'$.
That means there exists $\delta' >0$ such that $|v_{n}|_{\infty}\geq \delta'$ for all $n\in \mathbb{N}$.
The exponential decay property follows from a standard comparison arguments.
\end{proof}

\subsection{Concentration behavior} If $u_{\vr_{n}}$ is a solution of problem $(SCC1)$, then
${v_n}(x)=u_{\vr_{n}}(x+{y}_n)$ solves
\[
-\Delta v_n+{v_n}=\frac{1}{6-\mu}\Big(\int_{\R^3}  \frac{\widetilde{Q}_{\vr_n}(y)G(v_n(y))}{|x-y|^\mu}dy\Big)\widetilde{Q}_{\vr_n}(x)g({v_n})\quad  \hbox{in $\R^{3}$},
\]
with $\widetilde{Q}_{\vr_n}(x)=Q(\vr_n x+\vr_ny_n)$ and $({y}_n) \subset \R^{3}$ given in Lemma \ref{CP}. Moreover, up to a subsequence,
$$
v_n\rightarrow v \quad \mbox{in $E$} ,\quad   \tilde{y}_n\rightarrow y_0 \in \mathcal{Q},
$$
where $\tilde{y}_n=\vr_{n}{y}_n$. If $b_n$ denotes a maximum point of $v_n$, from Lemma \ref{0AI}, we know it is a bounded sequence in $\R^3$. Thus, there is $R>0$ such that
$b_n\in B_{R}(0)$. Thereby, the global maximum of $u_{{\vr}_{n}}$ is $z_{\vr_n}=b_n+{y}_n$ and
$$
\vr_{n}z_{\vr_{n}}=\vr_{n}b_{n}+\vr_{n}{y}_n=\vr_{n}b_{n}+\tilde{y}_{n}.
$$
From boundedness of $(b_n)$,  we get the limit
$
\displaystyle \lim_{n \to \infty}\vr_{n}z_{\vr_{n}}=y_0
$,
therefore
$
 \lim\limits_{ n \rightarrow\infty}Q(\vr_{n}z_{{\vr}_{n}})=Q(y_0)
$.
 We also point out that for any $\vr>0$  the sequence $(\vr z_\vr)$ is bounded, where $z_\vr$ is the maximum point of the solution $u_\vr$  obtained in Lemma \ref{Existence}. In fact, if there exists $\vr_j\to0 $ and $z_{\vr_j}$ of $u_{\vr_j}$ such that $\vr_jz_{\vr_j}\to \infty$. However, from the above arguments,
 $$
 \vr_jz_{\vr_j}=\vr_jb_{\vr_j}+\vr_j{y}_{\vr_j},
 $$
where ${y}_{\vr_j}$ is obtained in \eqref{J0} by non-vanishing argument with $(\vr_j{y}_{\vr_j})$ bounded, and $b_{\vr_j}$ is the maximum point of
of $v_{\vr_j}= u_{\vr_j}(x+{y}_{\vr_j})$. Consequently, $
 \vr_jz_{\vr_j}-\vr_j{y}_{\vr_j}=\vr_jb_{\vr_j}\to \infty.
 $
which contradicts with the fact $b_{\vr_j}$ lies in a ball $B_{R}(0)$.
 From Lemma \ref{Existence}, there is a positive solution for $(SCC1)$ for $\vr>0$ small enough. Therefore, the function ${w_\vr}(x)={u_\vr}(\frac{x}{\vr})$
is a positive solution of \eqref{SCP1}. Thus, the maximum points ${x_\vr}$ and ${z_\vr}$ of ${w_\vr}$ and ${u_\vr}$ respectively,
satisfy the equality ${x_\vr}=\vr{z_\vr}$. Setting $v_\vr(x):=w_\vr(\vr x+x_\vr)$, for any sequence $x_\vr\to x_0$, $\vr\to0$, it follows from Lemma \ref{CP} that,
$
\lim_{\vr\to0}\dist(x_\vr, \mathcal{Q})=0
$
 and $v_\vr$ converges in $E$ to a ground state solution $v$ of
    $$
-\Delta u +u  =\frac{\nu^2_{\max}}{6-\mu}\Big(\int_{\R^3} \frac{ G(u(y))}{|x-y|^\mu}dy\Big)g(u).
$$
From Lemma \ref{0AI}, for some $c,C>0$, $|w_\vr(x)|\leq C\exp\Big(-\frac{c}{\vr}|x-x_\vr|\Big).$

\section{Critical equation with linear potential}
Finally, to study the existence of solutions for the following equation
\begin{equation}
\label{SCC2}
-\Delta u+V(\vr x)u =\left(\int_{\mathbb{R}^3}\frac{|u(y)|^{6-\mu}+F(u(y))}{|x-y|^{\mu}}dy\right)\Big(|u|^{4-\mu}u+ \frac{1}{6-\mu}f(u)\Big)  \quad \mbox{in} \quad \R^3,
\tag{SCC2}
\end{equation}
we introduce the energy functional associated  to \eqref{SCP2} be $J_{\vr}$. The Nehari manifold associated to $J_{\vr}$ will be still denoted by ${\cal{N}_\vr}$, that is,
$$
{\cal{N}_\vr}=\Big\{u\in E:u\neq0,\langle J'_{\vr}(u), u\rangle=0\Big\}.
$$
Similar to Lemma \ref{mountain:1}, $J_{\vr}$ also satisfies the Mountain Pass Geometry and assumption $(f_3)$ implies that the least energy can be characterized by
\begin{equation*}
c_{\vr}=\inf_{u\in\cal{N}_\vr}J_{\vr}(u)=\inf_{u\in E\backslash\{0\}} \max_{t\geq 0},
J_{\vr}(tu),
\end{equation*}
moreover, there exists $\alpha>0$ which is a constant independent of $\vr$ such that $c_{\vr}>\alpha$.

\subsection{Compactness criteria}
Let $(u_n)$ be any $(PS)$ sequence of $J_{\vr}$ at $c$. Then, it is easy to see that $(u_n)$ is bounded and $c\geq0$.
Hence, without loss of generality, we may assume $u_n\rightharpoonup u$ in $E$ and
$u_n\to u$ in
$L^s_{{\rm loc}}(\mathbb{R}^3)$ for $1\leq s< 6$ and $u_n(x)\to u(x)$ a.e. for
$x\in\R^3$. Furthermore, arguing as in the proof of Theorem \ref{AQ}. we have the following lemma
\begin{lem}\label{splitting}
One has along a subsequence:
\begin{itemize}
\item[$(1).$] $J_{\vr}(u_n-u)\to c-J_{\vr}(u)$;
\item[$(2).$] $J'_{\vr}(u_n-u)\to 0$.
\end{itemize}
\end{lem}

\begin{lem}\label{PSlemma:1}
Suppose that $(f_1)$-$(f_3)$ and $(V)$ hold. Consider
a $(PS)_c$ sequence $(u_n)$ for $J_{\vr}$ with
$$
c<\frac{5-\mu}{2(6-\mu)}S_{H,L}^{\frac{6-\mu}{5-\mu}}.
$$
Suppose that $u_n\rightharpoonup
u$ in $E$. Then either $u_n\to u$ in $E$ along a subsequence or
$$
c-J_{\vr}(u)\geq m_{\kappa_{\infty},1,1},
$$
where $m_{\kappa_{\infty},1,1}$ is the minimax level of $\Phi_{{\kappa_{\infty},1,1}}$ given in (\ref{m}) with $\kappa=\kappa_{\infty},\mu=\tau=1$.
\end{lem}
\begin{proof}
Define $v_n=u_n-u$, from Lemma \ref{splitting} we know that $(v_n)$ is a $(PS)$ sequence at  $c-J_{\vr}(u)$ with $J_{\vr}(u)\geq0$. Now we suppose that $v_n\nrightarrow 0$ in $E$. From condition $(f_3)$, for each $u_n$ there is unique $t_n\in (0,\infty)$ such that $(t_nv_n)\subset {\cal N}_{{\kappa_\infty,1,1}}$. We divide the proof into three steps.

\vspace{0.2 cm}

\noindent
$\bullet$ {\bf Step 1.} The sequence $(t_n)$ satisfies
$$
\limsup_{n\to\infty}t_n\leq1.
$$
In fact, suppose by contradiction that the above claim does not hold. Then, there exist $\delta>0$ and a subsequence of $(t_n)$, still denoted by itself, such that
$$
t_n\geq 1+\delta\ \hbox{ for all}\ \ n\in \N.
$$
Since $\langle J'_{\vr}(v_n), v_n\rangle=o_n(1)$ and $(t_nv_n)\subset {\cal N}_{{\kappa_\infty,1,1}}$, we have
$$
 \|v_n\|^2_{\varepsilon}=\int_{\R^3}\int_{\R^3}\frac{G(v_n(y))g(v_n(x))v_n(x)}{|x-y|^{\mu}}dxdy+o_n(1)
$$
and
$$
t^2_n\int_{\mathbb{R}^3} (|\nabla v_n|^{2}+ \kappa_\infty |v_n|^{2})dx=\int_{\R^3}\int_{\R^3}\frac{G(t_nv_n(y))g(t_nv_n(x))t_nv_n(x)}{|x-y|^{\mu}}dxdy.
$$
Consequently,
$$
\int_{\mathbb{R}^3} (\kappa_\infty-V(\vr x))|v_n|^{2}dx+o_n(1)=\int_{\R^3}\int_{\R^3}\Big(\frac{G(t_nv_n(y))g(t_nv_n(x))t_nv_n(x)}{t^2_n|x-y|^{\mu}}-\frac{G(v_n(y))g(v_n(x))v_n(x)}{|x-y|^{\mu}}\Big)dxdy.
$$
Given $\xi>0$, from assumption $(V)$, there exists $R=R(\xi)>0$ such that
\begin{equation*}
V(\varepsilon x)\geq \kappa_\infty-\xi,\  \hbox{for any}\ |x|\geq R.
\end{equation*}
Using the fact that $v_n\to 0$ in $L^p(B_R(0))$, we conclude that
$$
\int_{\R^3}\int_{\R^3}\Big(\frac{G(t_nv_n(y))g(t_nv_n(x))t_nv_n(x)}{t^2_n|x-y|^{\mu}}-\frac{G(v_n(y))g(v_n(x))v_n(x)}{|x-y|^{\mu}}\Big)dxdy\leq\xi C+o_n(1),
$$
where $C=\sup_{n \in \mathbb{N}}|v_n|^2_2$. Notice that $(v_n)$  is $(PS)$ sequence at $c-J_{\vr}(u)$. We claim that there exist $(y_{n}) \subset \R^{3}$ and $r,\delta >0$ such that
$$
\int_{B_{r}(y_{n})}|v_{n}|^{2} dx\geq \delta, \quad n \in \mathbb{N}.
$$
Otherwise,
 \begin{center}
 $v_{n}\to 0$ in
 $L^s(\R^3)$, \ \ $2<s< 6$,\ \ $\hbox{as}\ \  n\to\infty$.
 \end{center}
 By repeating the arguments in Lemma \ref{Existence}, we have
\begin{equation}\label{DD12}
\begin{array}{ll}
\displaystyle\frac12 \|v_n\|^2_{\varepsilon}-\frac{1}{2(6-\mu)}\int_{\mathbb{R}^3}\int_{\mathbb{R}^3}   \frac{|v_n(x)|^{6-\mu}|v_n(y)|^{6-\mu}}{|x-y|^\mu}dxdy\to c-J_{\vr}(u)
\end{array}
\end{equation}
and
\begin{equation}\label{DD13}
 \|v_n\|^2_{\varepsilon}=\int_{\mathbb{R}^3}\int_{\mathbb{R}^3}   \frac{|v_n(x)|^{6-\mu}|v_n(y)|^{6-\mu}}{|x-y|^\mu}dxdy+o_{n}(1).
\end{equation}
Hence,
\begin{equation}\label{DD14}
 \|v_n\|^2_{\varepsilon}\leq \int_{\mathbb{R}^3}
\int_{\mathbb{R}^3}\frac{|v_{n}(x)|^{6-\mu}|v_{n}(y)|^{6-\mu}}
{|x-y|^{\mu}}dxdy+o_{n}(1)
\leq S_{H,L}^{\mu-6}\|v_{n}\|^{2(6-\mu)}+o_{n}(1).
\end{equation}
Since $ \|v_n\|_{\varepsilon}\nrightarrow0$, by \eqref{DD14} we get $ \|v_n\|_{\varepsilon}\geq  S_{H,L}^{\frac{6-\mu}{2(5-\mu)}}$. Then from \eqref{DD12}, \eqref{DD13} and \eqref{DD14} we easily conclude that
$$
c-J_{\vr}(u)\geq \frac{5-\mu}{2(6-\mu)}S_{H,L}^{\frac{6-\mu}{5-\mu}},
$$
which contradicts with our assumption that
$$
c<\frac{5-\mu}{2(6-\mu)}S_{H,L}^{\frac{6-\mu}{5-\mu}}.
$$
Thus there exists $(y_n)\subset \R^3$ and $r,\beta>0$ such that
$$
\int_{B_{r}(y_n)}|v_n|^{2}dx\geq \beta.
$$
If we define $\tilde{v}_n=v_n(x+y_n)$, we may suppose that, up to a subsequence, $\tilde{v}_n\rightharpoonup\tilde{v}$ in $E$. Moreover, using the fact that $v_n \geq 0$ for all $n \in \mathbb{N}$, there exists a subset $\Omega\subset \R^3$ with positive measure such that $\tilde{v}(x) >0 $ for all $x \in \Omega$. Consequently, from $(f_3)$, we get
$$
\aligned
&\int_{\Omega}\int_{\Omega}\frac{|\tilde{v}_n(y)||\tilde{v}_n(x)|}{|x-y|^{\mu}}\Big[\frac{G((1+\delta)\tilde{v}_n(y))g((1+\delta)\tilde{v}_n(x))(1+\delta)\tilde{v}_n(x)}{(1+\delta)|\tilde{v}_n(y)|(1+\delta)|\tilde{v}_n(x)|}-\frac{G(\tilde{v}_n(y))g(\tilde{v}_n(x))\tilde{v}_n(x)}{|\tilde{v}_n(y)||\tilde{v}_n(x)|}\Big]dxdy\\
&\hspace{4mm}=\int_{\Omega}\int_{\Omega}\Big[\frac{G((1+\delta)\tilde{v}_n(y))g((1+\delta)\tilde{v}_n(x))(1+\delta)\tilde{v}_n(x)}{(1+\delta)^2|x-y|^{\mu}}-\frac{G(\tilde{v}_n(y))g(\tilde{v}_n(x))\tilde{v}_n(x)}{|x-y|^{\mu}}\Big]dxdy\\
&\hspace{4mm}\leq\xi C+o_n(1)
\endaligned
$$
Letting $n\to \infty$ and applying Fatou's lemma, from the monotone assumption $(f_3)$, it follows
that
$$
0<\int_{\Omega}\int_{\Omega}\left(
\frac{G((1+\delta)\tilde{v}(y))g((1+\delta)\tilde{v}(x))(1+\delta)\tilde{v}(x)}{(1+\delta)^2|x-y|^{\mu}}-\frac{G(\tilde{v}(y))g(\tilde{v}(x))\tilde{v}(x)}{|x-y|^{\mu}}\right)dxdy\leq\xi C
$$
which is absurd, since the arbitrariness of $\xi $.

\vspace{0.5 cm}

\par
\noindent
$\bullet$ {\bf Step 2.} The sequence $(t_n)$ satisfies $$\displaystyle\limsup_{n\to\infty}t_n=1.$$
In this case, there exists a subsequence, still denoted by $(t_n)$, such that
$t_n\to1$. Since $m_{\kappa_{\infty},1,1}\leq\Phi_{\kappa_{\infty},1,1}(t_nv_n)$, we know
$$
c-J_{\vr}(u)+o_n(1)=J_{\vr}(v_n)\geq J_{\vr}(v_n)+m_{\kappa_{\infty},1,1}-\Phi_{\kappa_{\infty},1,1}(t_nv_n).
$$
Given $\xi>0$, from assumption $(V)$ there exists $R=R(\xi)>0$ such that
$$
V(\varepsilon x)\geq \kappa_{\infty}-\xi,\  \hbox{for any}\ |x|\geq R.
$$
Since
\begin{align*}
J_{\vr}(v_n)-\Phi_{\kappa_{\infty},1,1}(t_nv_n)
&=\frac{(1-t^2_n)}{2}\int_{\mathbb{R}^3} |\nabla v_n|^2dx+ \frac{1}{2}\int_{\mathbb{R}^3}V(\varepsilon x) |v_n|^2dx-\frac{t^2_n}{2}\int_{\mathbb{R}^N}\kappa_{\infty}|v_n|^2dx\\
& +\frac{1}{2(6-\mu)}\int_{\R^3}\int_{\R^3}\Big(\frac{G(t_nv_n(y))G(t_nv_n(x))}{t^2_n|x-y|^{\mu}}-\frac{G(v_n(y))G(v_n(x))}{|x-y|^{\mu}}\Big)dxdy,
\end{align*}
from the fact that $(v_n)$ is bounded in $E$ and $v_n\rightharpoonup0$,  we derive
$$
c-J_{\vr}(u)+o_n(1)=J_{\vr}(v_n)\geq m_{\kappa_{\infty},1,1}-\xi C+o_n(1),
$$
consequently,
$
c-J_{\vr}(u)\geq m_{\kappa_{\infty},1,1}.
$
\vspace{0.2 cm}

\noindent
$\bullet$ {\bf Step 3.} The sequence $(t_n)$ satisfies
$$
\limsup_{n\to\infty}t_n=t_0<1.
$$
\par
\noindent
We suppose that there exists a subsequence, still denoted by $(t_n)$, such that
$t_n\to t_0<1$.  Since $m_{\kappa_{\infty},1,1}\leq \Phi_{\kappa_{\infty},1,1}(t_nv_n)$ and $\langle\Phi'_{\kappa_{\infty},1,1}(t_nv_n),t_nv_n\rangle=0$, we get
$$
\aligned
m_{\kappa_{\infty},1,1}&\leq \Phi_{\kappa_{\infty},1,1}(t_nv_n)-\frac12\langle\Phi'_{\kappa_{\infty},1,1}(t_nv_n),t_nv_n\rangle\\
&=\frac{1}{2(6-\mu)}\int_{\R^3}\int_{\R^3}\frac{G(t_nv_n(y)g(t_nv_n(x))t_nv_n(x)-G(t_nv_n(y))G(t_nv_n(x))}{|x-y|^{\mu}}dxdy\\
&\leq \frac{1}{2(6-\mu)}\int_{\R^3}\int_{\R^3}\frac{G(v_n(y)g(v_n(x))v_n(x)-G(v_n(y))G(v_n(x))}{|x-y|^{\mu}}dxdy\\
&= J_{\vr}(v_n)-\frac12\langle J'_{\vr}(v_n), v_n\rangle
=c-J_{\vr}(u)+o_n(1).
\endaligned
$$
From this, the conclusion then follows.
\end{proof}

By an immediate consequence of Lemma \ref{splitting}, we have
\begin{lem}\label{PS}
Suppose $(f_1)$-$(f_3)$ and $(V)$ hold. Then $J_{\vr}$  satisfies $(PS)_c$ condition for all $c<m_{\kappa_{\infty},1,1}$.
\end{lem}

\begin{cor}\label{PSlemma:N}
Suppose $(f_1)$-$(f_3)$ and $(V)$. Then $J_{\vr}|_{\cal{N}_\vr}$  satisfies $(PS)_c$ condition for all $c<m_{\kappa_{\infty},1,1}$.
\end{cor}
\begin{proof}
Let $(u_n)\subset  {\cal{N}_\vr}$ be any sequence such that $J_{\vr}(u_n)\to
c$ and $\|J'_{\vr}(u_n)\|_*\to 0$.
Since $\cal{N}_\vr$ is a natural constraint, we know that $(u_n)$ is a $(PS)_c$ sequence with $c<m_{\kappa_{\infty},1,1}$. The conclusion follows from Lemma \ref{PS}.
\end{proof}

\section{Existence and multiplicity}
In this section, we are going to prove the existence and multiplicity of solutions. First we have the following existence result.

\begin{thm}[Existence of ground states]
	\label{GS}
Suppose that the nonlinearity $f$ satisfies $(f_1)$-$(f_3)$ and the potential function $V(x)$ satisfies condition $(V)$. Then, there exists $\vr_0>0$ such that problem $(SCC2)$ has a positive ground state solution $u_\vr$, for all $\vr<\vr_0$.
\end{thm}
\begin{proof}
It is easy to check that $J_{\vr}$ also satisfies the Mountain Pass geometry. Let
$$
m_\vr:=\inf_{u\in E\backslash\{0\}} \max_{t\geq 0}
J_{\vr}(tu)= \inf_{u\in \cal{N}_\vr}J_{\vr}(u).
$$
Then,  we know there exists a $(PS)$ sequence at $m_\vr$, i.e.,
$$
J'_{\vr}(u_n)\rightarrow0 \quad \mbox{and} \quad J_{\vr}(u_n)\rightarrow
{m_\vr}.
$$
Thus, by Lemma \ref{PS}, if $m_\vr<m_{\kappa_{\infty},1,1}$, then the existence of ground state solution is guaranteed. In what follows, we fix $\gamma>0$ and $\Psi_\gamma \in C^{\infty}_{0}(\R^{3})$ such that
$$
\Phi_{\kappa_{\min},1,1}(\Psi_\gamma)=\max_{t \geq 0}\Phi_{\kappa_{\min},1,1}(t \Psi_\gamma) \quad \mbox{and} \quad \Phi_{\kappa_{\min},1,1}(\Psi_\gamma) \leq m_{\kappa_{\min},1,1} - \gamma.
$$
By a direct computation,
$$
\limsup_{\epsilon \to 0} m_{\vr} \leq \Phi_{\kappa_{\min},1,1}(\Psi_\gamma) \leq m_{\kappa_{\min},1,1} - \gamma <m_{\kappa_{\infty},1,1}.
$$
Therefore, there is $\vr_0 >0$ such that $m_\vr <m_{\kappa_{\infty},1,1}$ for all $\vr \in  (0, \vr_0).$
\end{proof}

\noindent
Next, we are going to show the existence of multiple solutions
and study the behavior of their maximum points in relation to the set $\mathcal{V}$.
Let $\delta>0$ be fixed and $w$ be a ground state
solution of problem \eqref{A} with $A=\kappa_{\min}$. Define $\eta$ to be a smooth non-increasing cut-off function in $[0,\infty)$ such that $\eta(s)=1$ if $0\leq s\leq\frac{\delta}{2}$ and $\eta(s)=0$ if $s\geq\delta$.
 For any $y\in \mathcal{V}$, let us define
 $$
 \Psi_{\vr,y}(x)=\eta(|\vr x-y|)w\Big(\frac{\vr x-y}{\vr}\Big),
 $$
$t_\vr>0$ satisfying
$$
\max_{t\geq0}I_{\vr}(t\Psi_{\vr,y})=I_{\vr}(t_\vr\Psi_{\vr,y}),
$$
and $\Pi_\vr:\mathcal{V}\to \cal{N}_\vr$ by $\Pi_\vr(y)=t_\vr\Psi_{\vr,y}$. By construction, $\Pi_\vr(y)$ has compact support for any $y\in M$.

\begin{lem}\label{mapping2}
 The function $\Pi_{\vr}$ has the following limit
$$
\lim_{\vr \rightarrow 0}J_{\vr}(\Pi_{\vr}(y))=m_{\kappa_{\min},1,1}, \ \hbox{uniformly in}\ y\in \mathcal{V}.
$$
\end{lem}
\begin{proof}
	By contradiction, there exist $\delta_{0}>0$, $(y_{n})\subset \mathcal{V}$ and $\vr_{n}\rightarrow0$ such that
\begin{equation}\label{3.1}
|J_{\vr_n}(\Pi_{\vr_n}(y_{n}))-m_{\kappa_{\min},1,1}|\geq\delta_{0}.
\end{equation}
From Lebesgue's theorem,
$$
\lim_{n\rightarrow\infty}\int_{\mathbb{R}^{3N}}(|\nabla (t_{\vr_{n}}\Psi_{\vr_n,y_{n}})|^2+V(\vr_n x)|t_{\vr_{n}}\Psi_{\vr_n,y_{n}}|^2)dx=\int_{\mathbb{R}^{3}}(|\nabla w |^2+V_0|w|^2)dx
$$
and
$$
\lim_{n\rightarrow\infty}\int_{\R^3}\int_{\R^3}\frac{G(t_{\vr_{n}}\Psi_{\vr_n,y_{n}})G(t_{\vr_{n}} \Psi_{\vr_n,y_{n}})}{|x-y|^{\mu}}dxdy=\int_{\R^3}\int_{\R^3}\frac{G(w(y))G(w(x))}{|x-y|^{\mu}}dxdy.
$$
Since
$t_{\vr_{n}}\Psi_{\vr_n,y_n}\in{\cal{N}}_{\vr_n}$, it is easy to see the sequence $t_{\vr_{n}}\rightarrow1$. In fact, from the below equality
$$
t^2_{\vr_{n}}\int_{\mathbb{R}^{3}}|\nabla\Psi_{\vr_n,y_{n}}|^2+V(\vr_n x)|\Psi_{\vr_n,y_{n}}|^2dx=\frac{1}{6-\mu}\int_{\R^3}\int_{\R^3}\frac{G(t_{\vr_{n}}\Psi_{\vr_n,y_{n}})g(t_{\vr_{n}} \Psi_{\vr_n,y_{n}})t_{\vr_{n}}\Psi_{\vr_n,y_{n}}}{|x-y|^{\mu}}dxdy,
$$
we derive
$$
\|w\|_{\varepsilon}^2=\lim_{n\to \infty}\int_{\R^3}\int_{\R^3}\frac{G(t_{\vr_{n}}\Psi_{\vr_n,y_{n}})g(t_{\vr_{n}} \Psi_{\vr_n,y_{n}})t_{\vr_{n}}\Psi_{\vr_n,y_{n}}}{t^2_{\vr_{n}}|x-y|^{\mu}}dxdy.
$$
Now, using the fact that $w$ is a ground state solution of problem (\ref{A}) together with $(f'_4)$, we get that $t_{\vr_{n}}\rightarrow1$.
Now, note that
$$
\aligned
J_{\vr_{n}}(\Pi_{\vr_{n}}(y_{n}))&=\frac{t_{\vr_{n}}^{2}}{2}\int_{\mathbb{R}^{3}}|\nabla(\eta(|\vr_{n}x|)w(x))|^{2}dx+\frac{t_{\vr_{n}}^{2}}{2}\int_{\mathbb{R}^{3}}V(\vr_{n}x+y_{n})|(\eta(|\vr_{n}x|)w(x))|^{2}dx\\
&\hspace{5mm}-\frac{1}{2(6-\mu)}\int_{\R^3}\int_{\R^3}\frac{G(t_{\vr_{n}}\eta(|\vr_{n}y|)w(y))G(t_{\vr_{n}}\eta(|\vr_{n}x|)w(x))}{|x-y|^{\mu}}dxdy.
\endaligned
$$
Letting $n\rightarrow\infty$, we get $\lim_{n\rightarrow\infty}J_{\vr_{n}}(\Pi_{\vr_{n}}(y_{n}))=m_{\kappa_{\min},1,1}$, which contradicts with \eqref{3.1}.
\end{proof}

\noindent
For any $\delta>0$, let $\rho=\rho(\delta)>0$ be such that $M_\delta\subset B_\rho(0)$. Let $\chi:\R^3\to \R^3$ be defined as
$$
\chi(x):=x \,\,\, \quad \mbox{for $|x|\leq\rho$} \qquad \mbox{and} \qquad \chi(x):=\frac{\rho x}{|x|} \quad \mbox{for $|x|\geq\rho$}.
$$
Finally, let us consider $\beta_\vr:{\cal N}_\vr\to \R^3$ given by
$$
 \beta_\vr(u):= \frac{\displaystyle\int_{\mathbb{R}^{3}}\chi(\vr x)|u|^2dx}{\displaystyle\int_{\mathbb{R}^{3}}|u|^2dx}.
$$
Using the above notations,  by the Lebesgue's theorem permits to show the following lemma
\begin{lem}\label{noempty} The function $\Pi_{\vr}$ verifies
$$
\lim_{\vr \rightarrow 0}\beta_\vr(\Pi_{\vr}(y))=y, \quad \hbox{uniformly in}\ y\in \mathcal{V}.
$$
\end{lem}

\noindent
Let $h:\R^+\to\R^+$ be a positive function tending to $0$ such that $h(\vr)\to 0$ as $\vr\to0$ and let
$${\cal{\hat{N}}}_{\vr}:=\{u\in {\cal{N}}_{\vr}:J_{\vr}(u)\leq m_{\kappa_{\min},1,1}+h(\vr)\}.
$$
From Lemma \ref{noempty}, we know ${\cal{\hat{N}}}_{\vr}\neq\emptyset$.

\begin{lem}\label{mapping3}
Let $\delta>0$ and $ \mathcal{V}_\delta=\{x\in \R^3: \dist(x, \mathcal{V})\leq \delta\}$. Then
$$
\lim_{\vr\to0}\sup_{u\in {\cal{\hat{N}}}_{\vr}}\inf_{y\in  \mathcal{V}_\delta}|\beta_\vr(u)-y|=0.
$$
\end{lem}
\begin{proof}
Let $\vr_n \to 0$. For each $n\in \N$, there exists $(u_n)\subset {\cal{\hat{N}}}_{\vr_n}$, such that
$$
\inf_{y\in \mathcal{V}_\delta}|\beta_{\vr_n}(u_n)-y|=\sup_{u\in {\cal{\hat{N}}}_{\vr_n}}\inf_{y\in \mathcal{V}_\delta}|\beta_{\vr_n}(u)-y|+o_n(1).
$$
Since $(u_n)\subset {\cal{\hat{N}}}_{\vr_n} \subset {\cal{{N}}}_{\vr_n}$, it follows that
$
m_{\kappa_{\min},1,1}\leq m_{\vr_n}\leq J_{\vr_n}(u_n)\leq m_{\kappa_{\min},1,1}+h(\vr_n),
$
which means that
$$
J_{\vr_n}(u_n)\to m_{\kappa_{\min},1,1}\ \ \hbox{and} \ \ (u_n)\subset {\cal{N}}_{\vr_n}.
$$
 By repeating the arguments in Lemma \ref{Existence}, there exist $(y_{n}) \subset \R^{3}$ and $r,\delta >0$ such that
$$
\int_{B_{r}(y_{n})}|u_{n}|^{2}dx \geq \delta, \quad n \in \mathbb{N}.
$$
Setting  $v_{n}(x)=u_{n}(x+{y}_{n})$, up to a subsequence, if necessary, we can assume  $v_{n}\rightharpoonup v\not\equiv0$ in $E$. Let $t_{n}>0$ be such that $\tilde{v}_{n}=t_{n}v_{n}\in {\cal N}_{\kappa_{\min},1,1}$. Then,
$$
m_{\kappa_{\min},1,1}\leq \Phi_{\kappa_{\min},1,1}(\tilde{v}_{n})=\Phi_{\kappa_{\min},1,1}(t_{n}u_{n})\leq J_{\vr}(t_{n}u_{n})\leq J_{\vr}(u_{n})\to m_{\kappa_{\min},1,1}
$$
and so,
$$
\Phi_{\kappa_{\min},1,1}(\tilde{v}_{n})\rightarrow m_{\kappa_{\min},1,1}\ \hbox{and}\  (\tilde{v}_{n})\subset {\cal N}_{\kappa_{\min},1,1}.
$$
Then the sequence $(\tilde{v}_{n})$ is a minimizing sequence, and by Ekeland's variational principle \cite{MW}, we may also assume it is a bounded $(PS)$ sequence at $m_{\kappa_{\min},1,1}$. Thus, for some subsequence, $\tilde{v}_{n}\rightharpoonup \tilde{v}$ weakly in $E$ with $\tilde{v}\neq 0$ and $\Phi'_{\min,1,1}(\tilde{v})=0$. Then we can obtain that
$$
\Phi_{\kappa_{\min},1,1}(\tilde{v}_{n}-\tilde{v})\to 0,\  \   \Phi'_{\kappa_{\min},1,1}(\tilde{v}_{n}-\tilde{v})\to 0.
$$
Hence,
$$
\|\tilde{v}_{n}-\tilde{v}\|^2\leq C\lim_{n\to\infty}\big(\Phi_{\kappa_{\min},1,1}(\tilde{v}_{n}-\tilde{v})-\frac{1}{\theta(6-\mu)}\langle\Phi'_{\kappa_{\min},1,1}(\tilde{v}_{n}-\tilde{v}),\tilde{v}_{n}-\tilde{v}\rangle\big)=0,
$$
showing that $\tilde{v}_{n}\to\tilde{v}$ in $E$. Since $(t_n)$ is bounded, we can assume that for some subsequence $t_{n}\rightarrow t_{0}>0$, and so, $v_{n}\rightarrow v$ in $E$.
Now, we will show that $(\vr_{n}{y}_{n})$ has a subsequence satisfying $\vr_{n}{y}_{n}\rightarrow y\in  \mathcal{V}$. First we claim
$(\vr_{n}{y}_{n})$ is bounded in $\mathbb{R}^3$. Indeed, suppose by contradiction there exists a subsequence, still denoted by $(\vr_ny_{n})$, such that $|\vr_ny_{n}|\rightarrow\infty$. Since $\tilde{v}_{n}\rightarrow \tilde{v}$ in $E$ and $\kappa_{\min}<\kappa_{\infty}$, we have
$$
\aligned
m_{\kappa_{\min},1,1}&=\Phi_{\kappa_{\min},1,1}(\tilde{v})
<\Phi_{\kappa_{\infty},1,1}(\tilde{v})\\
&\leq\liminf_{n\rightarrow\infty}\left[\frac{1}{2}\int_{\mathbb{R}^{3}}|\nabla \tilde{v}_{n}|^{2}dx+\frac{1}{2}\int_{\mathbb{R}^{3}}V(\vr_{n}x+\vr_ny_{n})| \tilde{v}_{n}|^{2}dx-\frac{1}{2(6-\mu)}\int_{\R^3}\int_{\R^3}\frac{G(\tilde{v}_{n}(y))G(\tilde{v}_{n}(x))}{|x-y|^{\mu}}dxdy\right]\\
&=\liminf_{n\rightarrow\infty}\left[\frac{t_{n}^{2}}{2}\int_{\mathbb{R}^{3}}|\nabla u_{n}|^{2}dx+\frac{t_{n}^{2}}{2}\int_{\mathbb{R}^{3}}V(\vr_{n}x)| u_{n}|^{2}dx-\frac{1}{2(6-\mu)}\int_{\R^3}\int_{\R^3}\frac{G(t_{n}u_{n}(y))G(t_{n}u_n(x))}{|x-y|^{\mu}}dxdy\right]\\
&\leq \liminf_{n\rightarrow\infty}J_{\vr_n}(u_n)
=m_{\kappa_{\min},1,1},
\endaligned
$$
which does not make sense, showing that $(\vr_ny_{n})$ is bounded. Thus there exists a sequence $({y}_n)\subset \R^3$ such that $v_n(z)=u_n(x+{y}_n)$ has a convergent subsequence in $E$ and up to a subsequence, $\vr_{n}{y}_{n}\rightarrow y\in  \mathcal{V}$. Thus,
\begin{align*}
\label{BCF}
\beta_{\vr_n}(u_n)&=\frac{\displaystyle\int_{\mathbb{R}^{3}}\chi(\vr_n x)|u_n|^2dx}{\displaystyle\int_{\mathbb{R}^{3}}|u_n|^2dx}\\
&=\frac{\displaystyle\int_{\mathbb{R}^{3}}\chi(\vr_n x+\vr_ny_{n})|u_n(x+{y}_n)|^2dx}{\displaystyle\int_{\mathbb{R}^{3}}|u_n(x+{y}_n)|^2dx}
=\vr_ny_{n}+\frac{\displaystyle\int_{\mathbb{R}^{3}}[\chi(\vr_n x+\vr_ny_{n})-\vr_ny_{n}]|v_n(x)|^2dx}{\displaystyle\int_{\mathbb{R}^{3}}|v_n(x)|^2dx}\to y\in \mathcal{V}.
\end{align*}
Consequently, there exists $\vr_ny_n\in \mathcal{V}_\delta$ such that
$$
\lim_{n\to \infty}|\beta_{\vr_n}(u_n)-\vr_ny_n|=0,
$$
finishing the proof of the lemma.
\end{proof}

\begin{thm}[Multiplicity\ of\ solution]
	\label{MS}
Suppose that the nonlinearity $f$ satisfies $(f_1)-(f_3)$ and the potential function $V$ satisfies condition $(V)$. Then for any $\delta>0$ there exists $\vr_\delta>0$ such that problem $(SCC2)$ has at least $\cat_{\mathcal{V}_\delta} (\mathcal{V})$ positive solutions, $u_\vr$ for all $\vr<\vr_\delta$.
\end{thm}
\begin{proof}
 We fix a small $\vr> 0.$ Then , by Lemma \ref{mapping2} and \ref{mapping3},  $\beta_\vr\circ\Pi_{\vr}$ is homotopic to the inclusion map $id:\mathcal{V}\rightarrow \mathcal{V}_\delta$ and so,
$$
\cat_{\hat{N}_\vr}(\hat{N}_\vr)\geq \cat_{\mathcal{V}_\delta} (\mathcal{V}).
$$
Since that functional $J_{\vr}$ satisfies the $(PS)_{c}$ condition for $c\in (m_{\kappa_{\min},1,1},m_{\kappa_{\min},1,1}+h(\vr)),$ by the Lusternik-Schnirelman theory of
critical points \cite{MW}, we can conclude that $I_{\vr}$ has at least $cat_{\mathcal{V}_\delta} (\mathcal{V})$ critical points on ${\cal N}_\vr$. Consequently, $J_{\vr}$ has at least $\cat_{\mathcal{V}_\delta} (\mathcal{V})$ critical points in $E$.
\end{proof}

\noindent {\bf Concentration behavior.}  Let $\vr_n\to 0$ and $(u_n)$ be a sequence of solutions obtained in Lemma \ref{MS}, then there exists a sequence $({y}_n)\in \R^3$ such that $\vr_ny_{n}\to y\in \mathcal{V}$ and $v_n(z)=u_n(x+{y}_n)$ has a convergent subsequence in $E$. Similar to the arguments in Lemma \ref{0AI}, we know that there exists $C>0$ independent of $n$ such that $|v_{n}|_{{\infty}}\leq C$ and
$$
\lim_{|x|\rightarrow \infty}v_{n}(x)=0\ \hbox{uniformly in}\ n\in \mathbb{N}.
$$
Furthermore there exist $C$, $\beta>0$ such that
$|v_{n}(x)|\leq C \exp(-\beta|x|)$.
Similar to the analysis in section 3, by Theorem \ref{GS} and \ref{MS}, we know the existence and multiplicity of positive ground state solutions for equation $(SCC2)$ for $\vr>0$ small enough. Therefore, the function ${w_\vr}(x)={u_\vr}(\frac{x}{\vr})$
is a positive solution of \eqref{SCP2}. Thus, the maximum points ${x_\vr}$ and ${z_\vr}$ of ${w_\vr}$ and ${u_\vr}$ respectively,
satisfy the equality ${x_\vr}=\vr{z_\vr}$. Setting $v_\vr(x):=w_\vr(\vr x+x_\vr)$, for any sequence $x_\vr\to x_0$, $\vr\to0$, it follows from Lemma \ref{CP} that,
$$
\lim_{\vr\to0}\dist(x_\vr, \mathcal{V})=0
$$
 and $v_\vr$ converges in $E$ to a ground state solution $v$ of
    $$
-\Delta u +\kappa_{\min}u  =\frac{1}{6-\mu}\Big(\int_{\R^3} \frac{ G(u(y))}{|x-y|^\mu}dy\Big)g(u).
$$
Moreover, for some $c,C>0$, we have
$|w_\vr(x)|\leq C\exp\big(-\frac{c}{\vr}|x-x_\vr|\big).$

\section{Appendix: estimates}

\noindent
In $\R^3$, we know that
$$
U(x)=\frac{3^{1/4}}{(1+|x|^{2})^{1/2}}
$$
is a minimizer for $S$, the best Sobolev constant. By Proposition \ref{ExFu}, we know that $U(x)$ is also a minimizer for $S_{H,L}$.
Consider a cut-off function $\psi\in C_{0}^{\infty}(\R^3)$ such that
$$
\psi(x)=1,\quad |x|\leq \delta,\qquad
\psi(x)=0, \quad |x|\geq 2\delta,
$$
where $\delta>0$ is given in Lemma \ref{EMP1}. We define, for $\varepsilon>0$,
\begin{equation}\label{B17}
\aligned
u_{\varepsilon}(x):=\psi(x)U_{\varepsilon}(x),\quad\,\,\text{where
$U_{\varepsilon}(x) :=\varepsilon^{-1/2}U\left(\frac{x}{\varepsilon}\right)$}.
\endaligned
\end{equation}
Then, we have

\begin{lem}\label{ECT}
If $\frac{6-\mu}{2}<q<6-\mu$, then there holds
$$
\int_{\R^3}\int_{\R^3}\frac{|u_{\varepsilon}(x)|^{q}
|u_{\varepsilon}(y)|^{q}}
{|x-y|^{\mu}}dxdy
\geq \O(\varepsilon^{6-\mu-q})-\O(\varepsilon^{\frac{6-\mu}{2}}),
$$
and
$$
\int_{\R^3}\int_{\R^3}\frac{|u_{\varepsilon}(x)|^{6-\mu}|u_{\varepsilon}(y)|^{6-\mu}}
{|x-y|^{\mu}}dxdy
\geq C(3,\mu)^{\frac{3}{2}}S_{H,L}^{\frac{6-\mu}{2}}-\O(\varepsilon^{\frac{6-\mu}{2}}).
$$
\end{lem}
\begin{proof}

To estimate the convolution part, we know
\begin{equation}\label{D3}
\aligned
\int_{\R^3}\int_{\R^3}\frac{|u_{\varepsilon}(x)|^{q}
|u_{\varepsilon}(y)|^{q}}
{|x-y|^{\mu}}dxdy
&\geq \int_{B_{\delta}}\int_{B_{\delta}}\frac{|u_{\varepsilon}(x)|^{q}
|u_{\varepsilon}(y)|^{q}}{|x-y|^{\mu}}dxdy\\
&=\int_{B_{\delta}}\int_{B_{\delta}}\frac{|U_{\varepsilon}(x)|^{q}
|U_{\varepsilon}(y)|^{q}}{|x-y|^{\mu}}dxdy\\
&=\int_{B_{2\delta}}\int_{B_{2\delta}}\frac{|U_{\varepsilon}(x)|^{q}
|U_{\varepsilon}(y)|^{q}}{|x-y|^{\mu}}dxdy-2\int_{B_{2\delta}\setminus B_{\delta}}\int_{B_{\delta}}\frac{|U_{\varepsilon}(x)|^{q}|U_{\varepsilon}(y)|^{q}}{|x-y|^{\mu}}dxdy\\
&\hspace{12mm}-\int_{B_{2\delta}\setminus B_{\delta}}\int_{B_{2\delta}\setminus B_{\delta}}\frac{|U_{\varepsilon}(x)|^{q}|U_{\varepsilon}(y)|^{q}}{|x-y|^{\mu}}dxdy\\
&:=\A-2\B-\C,
\endaligned
\end{equation}
where
\begin{align*}
\A&:=\int_{B_{2\delta}}\int_{B_{2\delta}}\frac{|U_{\varepsilon}(x)|^{q}
|U_{\varepsilon}(y)|^{q}}{|x-y|^{\mu}}dxdy, \\
\B&:=\int_{B_{2\delta}\setminus B_{\delta}}\int_{B_{\delta}}\frac{|U_{\varepsilon}(x)|^{q}|U_{\varepsilon}(y)|^{q}}{|x-y|^{\mu}}dxdy, \\
\C&:=\int_{B_{2\delta}\setminus B_{\delta}}\int_{B_{2\delta}\setminus B_{\delta}}\frac{|U_{\varepsilon}(x)|^{q}|U_{\varepsilon}(y)|^{q}}{|x-y|^{\mu}}dxdy.
\end{align*}
We are going to estimate $\A$, $\B$ and $\C$. By direct computation, we know, for $\varepsilon<1$,
\begin{equation*}
\aligned
\A&\geq\varepsilon^{-q}C\int_{B_{\delta}}\int_{B_{\delta}}\frac{1}
{(1+|\frac{x}{\varepsilon}|^{2})^{\frac{q}{2}}|x-y|^{\mu}(1+|\frac{y}{\varepsilon}|^{2})^{\frac{q}{2}}}dxdy\\
&=C\varepsilon^{6-\mu-q}\int_{B_{\frac{\delta}{\varepsilon}}}\int_{B_{\frac{\delta}{\varepsilon}}}\frac{1}
{(1+|x|^{2})^{\frac{q}{2}}|x-y|^{\mu}
(1+|y|^{2})^{\frac{q}{2}}}dxdy\\
&\geq \O(\varepsilon^{6-\mu-q})\int_{B_{\delta}}\int_{B_{\delta}}\frac{1}
{(1+|x|^{2})^{\frac{q}{2}}|x-y|^{\mu}
(1+|y|^{2})^{\frac{q}{2}}}dxdy=\O(\varepsilon^{6-\mu-q}),
\endaligned
\end{equation*}
\begin{equation*}
\aligned
\B
&=\varepsilon^{q}C\int_{\Omega\setminus B_{\delta}}\int_{B_{\delta}}\frac{1}{(\varepsilon^{2}+|x|^{2})^{\frac{q}{2}}|x-y|^{\mu}(\varepsilon^{2}+|y|^{2})^{\frac{q}{2}}}dxdy\\
&\leq \O(\varepsilon^{q})\Big(\int_{\Omega\setminus B_{\delta}}\frac{1}
{(\varepsilon^{2}+|x|^{2})^{\frac{3q}{6-\mu}}}dx\Big)^{\frac{6-\mu}{6}}\Big(\int_{B_{\delta}}\frac{1}
{(\varepsilon^{2}+|y|^{2})^{\frac{3q}{6-\mu}}}dy\Big)^{\frac{6-\mu}{6}}\\
&=\O(\varepsilon^{\frac{6-\mu}{2}})\Big(\int_{0}^{\frac{\delta}{\varepsilon}}\frac{z^{2}}
{(1+z^{2})^{\frac{3q}{6-\mu}}}dz\Big)^{\frac{6-\mu}{6}}\\
&\leq \O(\varepsilon^{\frac{6-\mu}{2}})\Big(\int_{0}^{+\infty}\frac{z^{2}}
{(1+z^{2})^{\frac{3q}{6-\mu}}}dz\Big)^{\frac{6-\mu}{6}}
=\O(\varepsilon^{\frac{6-\mu}{2}}).
\endaligned
\end{equation*}
for each $q>\frac{6-\mu}{2}$ and
\begin{equation}\label{D6}
\aligned
\C
&=\varepsilon^{q}C\int_{B_{2\delta}\setminus B_{\delta}}\int_{B_{2\delta}\setminus B_{\delta}}\frac{1}{(\varepsilon^{2}+|x|^{2})^{\frac{q}{2}}|x-y|^{\mu}
(\varepsilon^{2}+|y|^{2})^{\frac{q}{2}}}dxdy\\
&\leq\varepsilon^{q}C\int_{B_{2\delta}\setminus B_{\delta}}\int_{B_{2\delta}\setminus B_{\delta}}\frac{1}{|x|^{q}|x-y|^{\mu}
|y|^{q}}dxdy
=\O(\varepsilon^{q}).
\endaligned
\end{equation}
From \eqref{D3}-\eqref{D6}, we have
\begin{equation*}
\aligned
\int_{\Omega}\int_{\Omega}\frac{|u_{\varepsilon}(x)|^{q}
|u_{\varepsilon}(y)|^{q}}
{|x-y|^{\mu}}dxdy
&\geq \O(\varepsilon^{6-\mu-q})-\O(\varepsilon^{\frac{6-\mu}{2}})-\O(\varepsilon^{q})\\
&=\O(\varepsilon^{6-\mu-q})-\O(\varepsilon^{\frac{6-\mu}{2}}).
\endaligned
\end{equation*}
Next, concerning the second assertion, we have
\begin{equation}\label{E7}
\aligned
\int_{\R^3}\int_{\R^3}&\frac{|u_{\varepsilon}(x)|^{6-\mu}|u_{\varepsilon}(y)|^{6-\mu}}
{|x-y|^{\mu}}dxdy\\
&\geq \int_{B_{\delta}}\int_{B_{\delta}}\frac{|u_{\varepsilon}(x)|^{6-\mu}|u_{\varepsilon}(y)|^{6-\mu}}{|x-y|^{\mu}}dxdy\\
&=\int_{\R^3}\int_{\R^3}\frac{|U_{\varepsilon}(x)|^{6-\mu}|U_{\varepsilon}(y)|^{6-\mu}}{|x-y|^{\mu}}dxdy
-2\int_{\R^3\setminus B_{\delta}}\int_{B_{\delta}}\frac{|U_{\varepsilon}(x)|^{6-\mu}|U_{\varepsilon}(y)|^{6-\mu}}{|x-y|^{\mu}}dxdy\\
&\hspace{7mm}-\int_{\R^3\setminus B_{\delta}}\int_{\R^3\setminus B_{\delta}}\frac{|U_{\varepsilon}(x)|^{6-\mu}|U_{\varepsilon}(y)|^{6-\mu}}{|x-y|^{\mu}}dxdy\\
&=C(3,\mu)^{\frac{3}{2}}S_{H,L}^{\frac{6-\mu}{2}}-2\D-\E,
\endaligned
\end{equation}
where
$$
\D=\int_{\R^3\setminus B_{\delta}}\int_{B_{\delta}}\frac{|U_{\varepsilon}(x)|^{6-\mu}|U_{\varepsilon}(y)|^{6-\mu}}{|x-y|^{\mu}}dxdy,\ \
\E=\int_{\R^3\setminus B_{\delta}}\int_{\R^3\setminus B_{\delta}}\frac{|U_{\varepsilon}(x)|^{6-\mu}|U_{\varepsilon}(y)|^{6-\mu}}{|x-y|^{\mu}}dxdy.
$$
By direct computation, we know
\begin{equation*}
\aligned
\D&=\int_{\R^3\setminus B_{\delta}}\int_{B_{\delta}}\frac{|U_{\varepsilon}(x)|^{6-\mu}|U_{\varepsilon}(y)|^{6-\mu}}{|x-y|^{\mu}}dxdy\\
&=\varepsilon^{6-\mu}C\int_{\R^3\setminus B_{\delta}}\int_{B_{\delta}}\frac{1}
{(\varepsilon^{2}+|x|^{2})^{\frac{6-\mu}{2}}|x-y|^{\mu}(\varepsilon^{2}+|y|^{2})^{\frac{6-\mu}{2}}}dxdy\\
&\leq \O(\varepsilon^{6-\mu})\left(\int_{\R^3\setminus B_{\delta}}\frac{1}
{(\varepsilon^{2}+|x|^{2})^{3}}dx\right)^{\frac{6-\mu}{6}}\left(\int_{B_{\delta}}\frac{1}{(\varepsilon^{2}+|y|^{2})^{3}}dy\right)^{\frac{6-\mu}{6}}\\
&\leq \O(\varepsilon^{6-\mu})\left(\int_{\R^3\setminus B_{\delta}}\frac{1}
{|x|^{6}}dx\right)^{\frac{6-\mu}{6}}\left(\int_{0}^{\delta}\frac{r^{2}}{(\varepsilon^{2}+r^{2})^{3}}dr\right)^{\frac{6-\mu}{6}}
\leq \O(\varepsilon^{\frac{6-\mu}{2}})
\endaligned
\end{equation*}
and
\begin{equation}\label{E9}
\aligned
\E&=\int_{\R^3\setminus B_{\delta}}\int_{\R^3\setminus B_{\delta}}\frac{|U_{\varepsilon}(x)|^{6-\mu}|U_{\varepsilon}(y)|^{6-\mu}}{|x-y|^{\mu}}dxdy\\
&=\varepsilon^{6-\mu}C\int_{\R^3\setminus B_{\delta}}\int_{\R^3\setminus B_{\delta}}\frac{1}
{(\varepsilon^{2}+|x|^{2})^{\frac{6-\mu}{2}}|x-y|^{\mu}(\varepsilon^{2}+|y|^{2})^{\frac{6-\mu}{2}}}dxdy\\
&\leq\varepsilon^{6-\mu}C\int_{\R^3\setminus B_{\delta}}\int_{\R^3\setminus B_{\delta}}\frac{1}
{|x|^{6-\mu}|x-y|^{\mu}|y|^{6-\mu}}dxdy
=\O(\varepsilon^{6-\mu}).
\endaligned
\end{equation}
It follows from \eqref{E7} to \eqref{E9}  that
\begin{align*}
\int_{\R^3}\int_{\R^3}\frac{|u_{\varepsilon}(x)|^{6-\mu}|u_{\varepsilon}(y)|^{6-\mu}}
{|x-y|^{\mu}}dxdy
&\geq C(3,\mu)^{\frac{3}{2}}S_{H,L}^{\frac{6-\mu}{2}}-\O(\varepsilon^{\frac{6-\mu}{2}})-\O(\varepsilon^{6-\mu})\\
&=C(3,\mu)^{\frac{3}{2}}S_{H,L}^{\frac{6-\mu}{2}}-\O(\varepsilon^{\frac{6-\mu}{2}}). \notag
\end{align*}
This concludes the proof.
\end{proof}

\vspace{1cm}
\noindent {\bf Acknowledgements.} \
The authors would like to thank  the anonymous referee
for his/her useful comments and suggestions which help to improve the presentation of the paper greatly.

\end{document}